\documentclass[11pt, oneside]{amsart}
\usepackage[text={5.58in,8.5in},centering,letterpaper,dvips]{geometry}
\usepackage[dvipsnames]{xcolor}
\usepackage{graphicx}
\usepackage{subcaption}
\usepackage{amsfonts}
\usepackage{epsf}
\usepackage{amssymb}
\usepackage{amsmath}
\usepackage{amscd}
\usepackage{tikz}
\usepackage{fancyhdr}
\usepackage{setspace}
\usepackage{hyperref}
\usepackage[all]{xy}
\usetikzlibrary{matrix}
\usepackage{verbatim}
\usepackage{enumerate}
\theoremstyle{theorem}
\newtheorem{theorem}{Theorem}[section]
\newtheorem{proposition}[theorem]{Proposition}
\newtheorem{lemma}[theorem]{Lemma}
\newtheorem{question}[theorem]{Question}
\newtheorem{corollary}[theorem]{Corollary}
\newtheorem{conjecture}[theorem]{Conjecture}

\theoremstyle{definition}

\newtheorem{remark}[theorem]{Remark}

\newcommand{\Z}{\mathbb{Z}}

\newcommand{\Q}{\mathbb{Q}}

\newcommand{\A}{\alpha}
\newcommand{\n}{\beta}
\newcommand{\g}{\gamma}
\newcommand{\X}{\times}
\newcommand{\pd}{\partial}

\newcommand{\emp}{\emptyset}
\newcommand{\T}{\mathcal T}
\newcommand{\CP}{\mathbb{CP}}

\makeatletter
\def\@seccntformat#1{%
  \protect\textup{\protect\@secnumfont
    \ifnum\pdfstrcmp{subsection}{#1}=0 \bfseries\fi% subsection # in \bfseries
    \csname the#1\endcsname
    \protect\@secnumpunct
  }%
}  
\makeatother

\makeatletter
\newtheorem*{rep@theorem}{\rep@title}
\newcommand{\newreptheorem}[2]{%
\newenvironment{rep#1}[1]{%
 \def\rep@title{#2 \ref{##1}}%
 \begin{rep@theorem}}%
 {\end{rep@theorem}}}
\makeatother

\newreptheorem{theorem}{Theorem}
\newreptheorem{lemma}{Lemma}
\newreptheorem{question}{Question}
\newreptheorem{corollary}{Corollary}
\newreptheorem{proposition}{Proposition}

\begin{document}

\rhead{\thepage}
\lhead{\author}
\thispagestyle{empty}

%\tableofcontents
%\listoffigures

\raggedbottom
\pagenumbering{arabic}
\setcounter{section}{0}

%%%%%%%%%%%%%%%%%%%%%%%%%%%%%%%%%%%%%%%%%%%%%%%%%%%%%%%%
%%%%%%%%%%%%%%%%%%%%%%%%%%%%%%%%%%%%%%%%%%%%%%%%%%%%%%%%
%%%%%%%%%%%%%%%%%%%%%%%%%%%%%%%%%%%%%%%%%%%%%%%%%%%%%%%%

\title{Manifolds with weakly reducible genus-three trisections are standard}
%Date{\today}

\author{Rom\'an Aranda}
\address{Department of Mathematics, University of Nebraska-Lincoln}
\email{jarandacuevas2@unl.edu}

\author{Alexander Zupan}
\address{Department of Mathematics, University of Nebraska-Lincoln}
\email{zupan@unl.edu}

\begin{abstract}
Heegaard splittings stratify 3-manifolds by complexity; only $S^3$ admits a genus-zero splitting, and only $S^3$, $S^1 \times S^2$, and lens spaces $L(p,q)$ admit genus-one splittings.  In dimension four, the second author and Jeffrey Meier proved that only a handful of simply-connected 4-manifolds have trisection genus two or less~\cite{MZ2}, while Meier conjectured that if $X$ admits a genus-three trisection, then $X$ is diffeomorphic to a spun lens space $S_p$ or its sibling $S_p'$, $S^4$, or a connected sum of copies of $\pm \mathbb{CP}^2$, $S^1 \times S^3$, and $S^2 \times S^2$~\cite{meier}.  We prove Meier's conjecture in the case that $X$ admits a weakly reducible genus-three trisection, where weak reducibility is a new idea adapted from Heegaard theory and is defined in terms of disjoint curves bounding compressing disks in various handlebodies.  The tools and techniques used to prove the main theorem borrow heavily from 3-manifold topology.  Of independent interest, we give a trisection-diagrammatic description of 4-manifolds obtained by surgery on loops and spheres in other 4-manifolds.
\end{abstract}

\maketitle

\section{Introduction}

The classification of closed, orientable manifolds in dimension $n$ is a foundational problem in geometric topology.  For $n=3$, the classification problem becomes much more manageable when we stratify 3-manifolds by Heegaard genus:  The only 3-manifold with a genus-zero Heegaard splitting is $S^3$, and the only 3-manifolds with a genus-one Heegaard splitting are $S^3$, $S^1 \X S^2$, and the family of lens spaces $L(p,q)$.  Turning to dimension four, we may stratify manifolds by their trisection genus, a higher-dimensional analogue of Heegaard genus.  In their seminal work, Gay and Kirby showed the only smooth 4-manifold admitting a genus-zero trisection is $S^4$, while the only manifolds admitting genus-one trisections are $S^4$, $S^1 \X S^3$, and $\pm \CP^2$~\cite{GK}.  This classification program was pushed further by Meier and the second author, who proved

\begin{theorem}\cite{MZ2}\label{thm:MZ}
If $X$ admits a genus-two trisection $\T$, then either $\T$ is reducible and $X$ can be expressed as a connected sum of genus-one 4-manifolds, or $\T$ is the standard trisection of $S^2 \X S^2$.
\end{theorem}

While the genus-one classification is a straightforward exercise, the genus-two classification relies on intricate combinatorics and the Wave Theorem of~\cite{HOT}, a deep result from the theory of Heegaard splittings.  Regarding genus-three trisections, Meier showed that Pao's manifolds $S_p$ and $S_p'$ (obtained by surgery on a loop in $S^1 \X S^3$) admit genus-three trisections.  He also conjectured

\begin{conjecture}\cite{meier}\label{conj:meier}
If $X$ admits a genus-three trisection, then $X$ is diffeomorphic to a spun lens space $S_p$ or its sibling $S_p'$, $S^4$, or a connected sum of copies of $\pm \CP^2$, $S^1 \X S^3$, and $S^2 \X S^2$.
\end{conjecture}

In the present paper, we prove that any 4-manifold $X$ with a weakly reducible genus-three trisection satisfies Meier's conjecture, where relevant definitions are included below.

\begin{theorem}\label{thm:main}
Suppose $X$ admits a weakly reducible genus-three trisection $\T$.  Then either $\T$ is reducible, or $\T$ contains a five-chain.  In particular, $X$ is diffeomorphic to a spun lens space $S_p$ or its sibling $S_p'$, $S^4$, or a connected sum of copies of $\pm \CP^2$, $S^1 \X S^3$, and $S^2 \X S^2$.
\end{theorem}

A trisection $\T$ is determined up to diffeomorphism by its \emph{spine}, the union of three \mbox{3-dimensional} handlebodies $H_{\A} \cup H_{\n} \cup H_{\g}$.  A trisection $\T$ is said to be \emph{reducible} if there is a curve $c$ bounding disks in all three handlebodies; in this case, we can split $\T$ into a connected sum of smaller-genus trisections.  In a foundational result from Heegaard theory, Casson and Gordon proved that if a 3-manifold $Y$ admits a weakly reducible Heegaard splitting, then either the splitting is reducible or $Y$ contains an essential surface~\cite{CG}.  We adapt weak reducibility to the setting of trisections, providing further evidence that key ideas from Heegaard theory can inform new insights in dimension four via trisections.  We say that $\T$ is \emph{weakly reducible} if there are disjoint non-separating curves $c$ and $c'$ such that $c$ bounds a disk in one of the three handlebodies, and $c'$ bounds a disk in the other two.  See~\cite{ST} for additional discussion of weak reduction in the context of Heegaard splittings.

A \emph{five-chain} is a specific sequence of five curves bounding disks in each of the three handlebodies such that consecutive curves intersect once and other pairs of curves are disjoint.  A more rigorous definition can be found in Section~\ref{sec:fivechain}, in which we show that if $\T$ admits a five-chain, then $X$ can be obtained by surgery on a loop $\ell$ in a manifold $X'$ admitting a lower-genus trisection $\T'$.  The proof of Theorem~\ref{thm:main} relies heavily on ideas from the theory of Heegaard splittings and uses a variety of tools and techniques, including connections between 4-manifolds and Dehn surgery on knots in 3-manifolds, weak reduction and thin position of Heegaard splittings, and many ideas arising in the proof of Theorem~\ref{thm:MZ}.

In the spirit of classifying as many genus-three trisections as possible, we prove a stronger version of Theorem~\ref{thm:main} using the same techniques.  Given a trisection $\T$ for $X$ with spine $H_{\A} \cup H_{\n} \cup H_{\g}$, a \emph{dependent triple} for $\T$ consists of three pairwise disjoint non-separating curves $\A_1$, $\n_1$, and $\g_1$ bounding disks in $H_{\A}$, $H_{\n}$, and $H_{\g}$, respectively, such that the homology classes $[\A_1]$, $[\n_1]$, and $[\g_1]$ are linearly dependent in $H_1(\Sigma)$.

\begin{theorem}\label{thm:main2}
Suppose a genus-three trisection $\T$ of $X$ admits a dependent triple.  Then either $\T$ is reducible, or $\T$ contains a five-chain.  In particular, $X$ is diffeomorphic to a spun lens space $S_p$ or its sibling $S_p'$, $S^4$, or a connected sum of copies of $\pm \CP^2$, $S^1 \X S^3$, and $S^2 \X S^2$.
\end{theorem}

\begin{corollary}\label{cor:sc}
If $X$ is simply-connected and admits a genus-three trisection $\T$ with a dependent triple, then $X$ is diffeomorphic to $S^4$ or a connected sum of copies of $\pm \CP^2$, and $S^2 \X S^2$.
\end{corollary}

Corollary~\ref{cor:sc} is notable in light of the trisection genus additivity problem, which asks whether $g(X_1 \# X_2) = g(X_1) + g(X_2)$.  In particular, there are examples of exotic $\CP^2 \#^2 \overline{\CP}^2$ (manifolds $X$ homeomorphic but not diffeomorphic to $\CP^2 \#^2 \overline{\CP}^2$)~\cite{AP,FS}.  If trisection genus is additive, then the fact that there is some $k$ such that $X \#^k (S^2 \X S^2)$ is diffeomorphic to $\CP^2 \#^2 \overline{\CP}^2 \#^k(S^2 \X S^2)$~\cite{wall} implies that $g(X) = g(\CP^2 \#^2 \overline{\CP}^2) = 3$.  Corollary~\ref{cor:sc} can be interpreted as evidence that no exotic $\CP^2 \#^2 \overline{\CP}^2$ admits a genus-three trisection, which would imply that trisection genus is not additive. For more information on the additivity of trisection genus and exotic pairs, see Section 1.3 of~\cite{LC_M_Rational}.

\begin{remark}
The only progress on the genus-three classification problem appears in work of the first author and Moeller in~\cite{AM}.  In Section 7, they prove that a special class of genus-three trisections, called \emph{Farey trisections}, correspond to standard 4-manifolds (and thus satisfy Meier's conjecture).  Farey trisections are examples of genus-three diagrams with a dependent triple; thus we recover Theorem 7.2 of~\cite{AM}.
\end{remark}

The order of the paper is as follows:  In Section~\ref{sec:prelim}, we dispense with the necessary preliminaries.  In Section~\ref{sec:technical}, we introduce technical lemmas from 3-manifold topology, citing some and proving others which do not appear elsewhere in the literature.  These lemmas are the main input used to prove Proposition~\ref{prop:dp}, the starting point for our classification of weakly reducible genus-three trisections.  Section~\ref{sec:triple} discusses some results on \emph{Heegaard triples}, a generalization of trisection diagrams, following~\cite{MSZ} and~\cite{MZ2} to prove generalizations of the main theorems from those papers.  In Section~\ref{sec:fivechain}, we introduce the notion of a five-chain and surgery on a five-chain, and finally, in Section~\ref{sec:proof}, we put everything together to prove Theorem~\ref{thm:main}.  Section~\ref{sec:depend} extends these ideas to dependent triples.  The paper concludes with questions for future exploration in Section~\ref{sec:question}.

\textbf{Acknowledgements:}  We are grateful to Gabe Islambouli and Maggie Miller for their help at the outset of this project, and we thank Jeffrey Meier for comments on a draft of this paper and for conversations related to this work and to the work in~\cite{meier}.  AZ is supported by NSF award DMS-2405301 and a Simons Foundation Travel Award.

\section{Preliminaries}\label{sec:prelim}

All manifolds are smooth and orientable (unless otherwise specified).  If $Y \subset X$, we let $\eta(Y)$ denote a regular neighborhood of $Y$ in $X$, and we let $X \setminus Y = X - \eta(Y)$.  A $n$-dimensional \emph{1-handlebody} of genus $g$ is the union of an $n$-dimensional zero handle and $g$ $n$-dimensional 1-handles.  Following convention, we refer to a 3-dimensional 1-handlebody simply as a \emph{handlebody}.  A \emph{compression-body} $C$ is obtained from the product $\Sigma \X I$, where $\Sigma$ is a connected, closed surface of genus $g > 0$, by attaching 3-dimensional 2-handles to $\Sigma \X \{0\}$ and capping off any resulting 2-sphere boundary components with a 3-ball.  We let $\pd_+ C = \Sigma \X \{1\}$ and $\pd_- C = \pd C \setminus \pd_+ C$.  As a consequence, a handlebody is a compression-body $C$ such that $\pd_- C = \emp$.

By a \emph{curve} in a surface $\Sigma$, we mean a free homotopy class of an essential simple closed curve.  A \emph{compressing disk} $D$ for a compression-body $C$ is a properly embedded disk $D \subset C$ such that $\pd D$ is a curve in $\Sigma = \pd_+ C$.  It is a standard fact that $D$ is determined up to isotopy by its boundary $\pd D$, which we call a \emph{compressing curve} for $C$.  A \emph{cut system} for a genus-$g$ handlebody $H$ is a collection $\A$ of pairwise disjoint compressing curves such that $\pd H \setminus \A$ is a connected planar surface.  In this case, $\A$ must contain exactly $g$ curves, and we write $H = H_{\A}$.  Two cut systems $\A$ and $\A'$ determine the same handlebody $H$ if and only if $\A$ and $\A'$ are related by a finite sequence of handleslides in $\Sigma$~\cite{johan}.  A curve $c'$ in $\Sigma$ is called \emph{primitive} in $H$ if there exists a compressing curve $c$ for $H$ such that $|c \cap c'| = 1$.

Given a curve $c$ in $\Sigma$, let $\Sigma_c$ denote the surface obtained by capping off the boundary components of $\Sigma \setminus c$ with two disks, $D_0$ and $D_1$.  We say that $\Sigma_c$ is obtained by \emph{compression along $c$}, and we call the glued-in disks $D_0$ and $D_1$ the \emph{scars} of the compression.  If we compress $\Sigma$ along multiple disjoint curves, say $c$ and $c'$, we denote the resulting closed surface by $\Sigma_{c,c'}$.  Note that if $c$ is a compressing curve for $H$ and $D$ is a disk in $H$ bounded by $c$, then we can view $\Sigma_c$ as $\pd(H \setminus D)$.

\begin{remark}\label{rmk:slide}
Any curve $c_1$ in $\Sigma_c$ disjoint from the scars of compression can be viewed as a curve in $\Sigma$ as well, also there is an important subtlety to point out:  If $c_1$ and $c_2$ are homotopic curves in $\Sigma_c$ disjoint from the scars of compression, is does not necessarily imply that $c_1$ and $c_2$ are homotopic in $\Sigma$, but rather that $c_1$ becomes homotopic to $c_2$ in $\Sigma$ after some number of handleslides over the curve $c$ in $\Sigma$ (where each handleslide corresponds precisely to a homotopy of $c_1$ in $\Sigma_c$ that passes $c_1$ over a scar).
\end{remark}

Every closed 3-manifold $Y$ admits a \emph{Heegaard splitting}, a decomposition $Y = H_1 \cup_{\Sigma} H_2$, where $H_1$ and $H_2$ are handlebodies and $\Sigma = \pd H_1 = \pd H_2$.  In this case, the splitting can be represented by a \emph{Heegaard diagram} $(\Sigma;\A,\n)$, where $\A$ and $\n$ are cut systems for $H_1$ and $H_2$, respectively.  A curve $c$ in $\Sigma$ that is a compressing curve for both $H_1$ and $H_2$ is called a \emph{reducing curve}.  In this case, $c$ bounds disks $D_1 \subset H_1$ and $D_2 \subset H_2$, so that $S = D_1 \cup D_2$ is a 2-sphere that meets $\Sigma$ in a single curve.  If $c$ is separating, we can split $Y$ into $Y_1 \# Y_2$ along $S$, and we can cut $\Sigma$ along $c$ to obtain Heegaard surfaces for $Y_1$ and $Y_2$.  If $c$ is nonseparating, then so is $S$, and we can see that $Y = Y' \# (S^1 \X S^2)$ via a standard argument, in which $S$ is isotopic to $\{\text{pt}\} \X S^2$ in the $S^1 \X S^2$ summand.  If there are compressing curves $c_1$ and $c_2$ for $H_1$ and $H_2$, respectively, such that $|c_1\cap c_2| = 1$, we say the splitting is \emph{stabilized}.  In this case, the Heegaard splitting has a genus-one $S^3$ summand, and we can find a lower-genus Heegaard splitting of $Y$, called a \emph{destabilization} of our original splitting.

Standard diagrams come up in a number of places in our arguments, and so we define them here.  A genus-$g$ Heegaard diagram $(\Sigma;\A,\n)$ for $\#^k (S^1 \X S^2)$ is called \emph{standard} if $|\A_i \cap \n_j| = \delta_{ij}$ for $1 \leq i,j \leq g-k$, $\A_i = \n_i$ for $g-k < i \leq g$, and all other pairs of curves in $\A$ and $\n$ are disjoint.  Waldhausen's Theorem implies that every Heegaard splitting of $\#^k(S^1 \X S^2)$ admits a standard diagram~\cite{waldhausen} (see also~\cite{schleimer}).  However, we observe that a given splitting does not have a \emph{unique} standard diagram; see Figure~\ref{fig:standard} for an example.  A genus-2 Heegaard diagram $(\Sigma;\A,\n)$ is called \emph{standard} if there exists a separating reducing curve $c$ for the splitting disjoint from $\A \cup \n$.  In this case, the Heegaard diagram can be reduced to the connected sum of two genus-one diagrams.  These two notions of standardness agree, but the second definition extends the idea of standardness to other genus-two manifolds such as $L(p,q) \# (S^1 \X S^2)$.

\begin{figure}[h!]
  \centering
  \includegraphics[width=.45\linewidth]{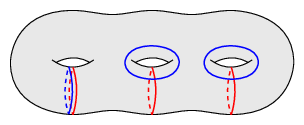} \qquad
    \includegraphics[width=.45\linewidth]{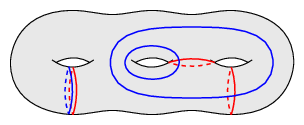} 
  \caption{Two examples of standard Heegaard diagrams for a genus-three Heegaard splitting of $S^1 \X S^2$}
	\label{fig:standard}
\end{figure}

Heegaard theory is also well-developed in the setting of compact 3-manifolds:  A Heegaard splitting for a compact 3-manifold $Y$ is a decomposition $Y = C_1 \cup_{\Sigma} C_2$, where $C_1$ and $C_2$ are compression-bodies with $\Sigma = \pd_+ C_1 = \pd_+ C_2$ and $\pd Y = \pd_- C_1 \cup \pd_- C_2$.  As above, we say that the Heegaard splitting is \emph{reducible} if there exists a compressing curve $c$ in $\Sigma$ for both $C_1$ and $C_2$.  A classical result known as Haken's Lemma asserts

\begin{lemma}\cite{haken}\label{lem:haken}
If $Y = C_1 \cup_{\Sigma} C_2$ is a Heegaard splitting of a reducible 3-manifold $Y$, then $\Sigma$ is reducible.
\end{lemma}

An important and technical tool in our work is a generalized Heegaard splitting.  A \emph{generalized Heegaard splitting} for $Y$ is the expression of $Y$ as
\[ Y = C_1^1 \cup_{\Sigma_1} C_2^1 \cup_{S_1} C_1^2 \cup_{\Sigma_2} \cup \dots \cup_{S_{n-1}} C_1^n \cup_{\Sigma_n} \cup C_2^n,\]
where
\begin{enumerate}
\item each $C_1^i$ and $C_2^i$ is a compression-body,
\item $\Sigma_i = \pd_+ C_1^i = \pd_+ C_2^i$,
\item $S_i = \pd_- C_2^i = \pd_- C_1^{i+1}$, and
\item $\pd Y = \pd_- C_1 \cup \pd_- C_2^n$.
\end{enumerate}
We call the surfaces $\Sigma_i$ \emph{thick surfaces} and $S_i$ \emph{thin surfaces} of the generalized Heegaard splitting, and we will abuse notation and denote a generalized splitting by the sequence of thick/thin surfaces $(\Sigma_1,S_1,\dots,S_{n-1},\Sigma_n)$ in $Y$, which uniquely determines the generalized Heegaard splitting.  A Heegaard splitting $Y = C_1 \cup_{\Sigma} C_2$ is called \emph{weakly reducible} if there exist disjoint curves $c_1$ and $c_2$ bounding compressing disks in $C_1$ and $C_2$, respectively, and $c_1$ and $c_2$ are called a \emph{weak-reducing pair}.  Generalized Heegaard splittings arise naturally from weakly reducible Heegaard splittings:  Suppose that $\Sigma$ is a weakly reducible Heegaard surface for $Y$ with a weak-reducing pair $c_1$ and $c_2$.  In a process known as \emph{untelescoping}, we can obtain a generalized Heegaard splitting $(\Sigma_{c_1},\Sigma_{c_1,c_2},\Sigma_{c_2})$ for $Y$.  In this case, the compression-body $C_2^1$ can be obtained by attaching a 2-handle to $\Sigma_{c_1}$ along $c_2$ (and, in parallel, $C_1^2$ is obtained by attaching a 2-handle to $\Sigma_{c_2}$ along $c_1$).  A detailed treatment of generalized Heegaard splittings can be found in~\cite{SSS}.

\begin{remark}
In this paper, we will deal exclusively with genus-three weakly reducible Heegaard splittings $Y = C_1 \cup_{\Sigma} C_2$.  Suppose that $c_1$ and $c_2$ are a weak-reducing pair such that both $c_1$ and $c_2$ are non-separating curves in $\Sigma$.  There are two qualitatively different cases:  If $c_1$ and $c_2$ are mutually non-separating, the induced thin surface $\Sigma_{c_1,c_2}$ is a torus.  If $c_1$ and $c_2$ are mutually separating, $\Sigma_{c_1,c_2}$ is the disjoint union of two tori.  A schematic diagram of untelescoping in either case is shown in Figure~\ref{fig:untel}.
\end{remark}

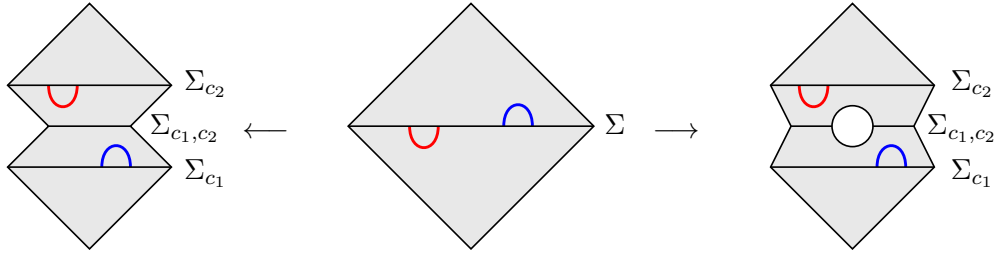
\begin{figure}[h!]
  \centering
  \includegraphics[width=.25\linewidth]{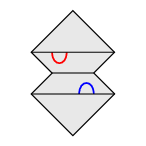} \raisebox{1.75cm}{$\longleftarrow$} \quad
  \includegraphics[width=.25\linewidth]{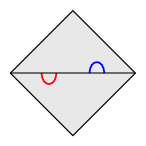} \quad \raisebox{1.75cm}{$\longrightarrow$}
  \includegraphics[width=.25\linewidth]{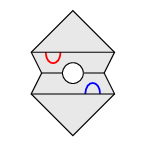} 
\put (-148,51) {$\Sigma$}
\put (-320,51) {$\Sigma_{c_1,c_2}$}
\put (-307,68) {$\Sigma_{c_2}$}
\put (-307,34) {$\Sigma_{c_1}$}
\put (-26,51) {$\Sigma_{c_1,c_2}$}
\put (-17,68) {$\Sigma_{c_2}$}
\put (-17,34) {$\Sigma_{c_1}$}
\caption{Schematic diagrams for the two different untelescopings of a weakly reducible genus-three Heegaard splitting.}
	\label{fig:untel}
\end{figure}

As with Heegaard splittings above, we say that a thick surface $\Sigma_i$ is \emph{weakly reducible} if there exist disjoint compressing curves $c_1$ and $c_2$ in $\Sigma_i$ for $C_1^i$ and $C_2^i$, respectively.  We state a useful proposition, which is foundational to the theory of thin position of 3-manifolds.

\begin{proposition}\cite{CG,ST}\label{prop:thin}
Suppose $(\Sigma_1,S_1,\dots,S_{n-1},\Sigma_n)$ is a generalized Heegaard splitting of $Y$.  If a thin surface $S_i$ is compressible, then at least one thick surface $\Sigma_j$ is weakly reducible.
\end{proposition}

In~\cite{GK}, Gay and Kirby showed that every closed smooth 4-manifold $X$ admits a trisection $\T$, a decomposition $X = X_1 \cup X_2 \cup X_3$ in which each $X_i$ is a 4-dimensional 1-handlebody $\natural^{k_i} (S^1 \X D^3)$, and each intersection $H_i = X_i \cap X_{i+1}$ is a (3-dimensional) handlebody.  In this case, it follows that the triple intersection $\Sigma = X_1 \cup X_2 \cup X_3$ is a genus-$g$ surface.  When we wish to emphasize the complexity of the components of a trisection, we will say that $\T$ is a $(g;k_1,k_2,k_3)$-trisection.  In analogy to the 3-dimensional situation, a trisection is determined up to diffeomorphism by a \emph{trisection diagram} $(\Sigma;\A,\n,\g)$, in which $\A$, $\n$, and $\g$ are cut systems for $H_1$, $H_2$, and $H_3$, respectively, which we will usually denote $H_{\A}$, $H_{\n}$, and $H_{\g}$.

As in the case of Heegaard splittings, a trisection is said to be \emph{reducible} if there exists a curve $c$ in $\Sigma$ which is a compressing curve for all three of $H_{\A}$, $H_{\n}$, and $H_{\g}$, called a \emph{reducing curve}.  In this case, if $c$ is separating, we can we can express $\T$ as a connected sum $\T = \T' \# \T''$, where $X = X' \# X''$ and $\T'$ and $\T''$ are trisections of $X'$ and $X''$, respectively.  If $c$ is nonseparating, then a standard argument shows that $X= X' \# (S^1 \X S^3)$, and $\T$ can be decomposed as the connected sum of a trisection $\T'$ of $X'$ and the standard genus-one trisection of $S^1 \X S^3$.  If there exists a curve $c$ that compresses in $H_{\A}$ and a curve $c'$ that compresses in both $H_{\n}$ and $H_{\g}$ such that $|c \cap c'|$, we say that $\T$ is \emph{stabilized}.  If $\T$ is stabilized, then $\T$ is reducible and we can split off a genus-one trisection of $S^4$.  In other words, we can find a lower-genus trisection for the same 4-manifold.

Motivated by the definition of weak reducibility from Heegaard theory, we introduce a similar concept here:  We say that $\T$ is \emph{weakly reducible} if there exist non-separating curves $c$ and $c'$ in $\Sigma$ such that $c$ is a compressing curve for $H_{\A}$ and $c'$ is a compressing curve for both $H_{\n}$ and $H_{\g}$.  An example of a weakly reducible trisection can be seen in Figure~\ref{fig:spunrp2}.

\begin{figure}[h!]
  \centering
  \includegraphics[width=.5\linewidth]{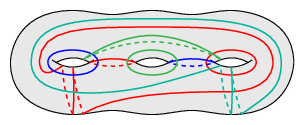}
  \caption{A genus-three trisection diagram $(\Sigma;\A,\n,\g)$ for the manifold $S_2$, the spin of $\mathbb{RP}^3$, in which curves in $\A$ are red, curves in $\n$ are blue, and curves in $\g$ are green.  Note that the teal curve belongs to both $\n$ and $\g$ and is disjoint from a curve in $\A$, so this trisection is weakly reducible.}
	\label{fig:spunrp2}
\end{figure}

\begin{remark}
The reader may wonder why we require the compressing curves in the definition of weak reducibility to be non-separating, because this is not a requirement for a weakly reducible Heegaard splitting.  However, given a pair $c$ and $c'$ of weak-reducing curves for a Heegaard surface, we can always replace $c$ and $c'$ with non-separating curves by standard arguments.  Thus, these definitions do align, although for trisections we add the non-separating requirement to the definition, since it is not a consequence of the less restrictive definition as it is in dimension three.
\end{remark}

In our exploration of genus-three trisections, we can rule out certain complexity parameters using a known classification.

\begin{proposition}\cite{MSZ}\label{prop:g-1}
Suppose that $\T$ is a $(g;k_1,k_2,k_3)$-trisection, where $k_i \geq g-1$ for some $i$.  Then $\T$ is reducible and can be expressed as the connected sum of genus-one trisections.
\end{proposition}

As a consequence, we will henceforth assume that any $(3;k_1,k_2,k_3)$-trisection appearing in this paper will satisfy $k_i \in \{0,1\}$, since these are the unclassified cases.

The last topic in this section is surgery on a loop.  Let $X$ be a closed 4-manifold, and let $\ell$ be a loop in $X$; that is, an embedded $S^1$.  Consider a 4-manifold $X_{\ell} = (X \setminus \ell) \cup (S^2 \X D^2)$.  We say that $X_{\ell}$ is obtained by \emph{surgery on $\ell$}.  Up to diffeomorphism, there are at most two ways to cap off $S^1 \X S^2 = \pd(X \setminus \ell)$ with $S^2 \X D^2$, and so there are at most two distinct 4-manifolds $X_{\ell}$ and $X'_{\ell}$ obtained by surgery on the loop $\ell$ in $X$~\cite{gluck}.  The 4-manifolds $X_{\ell}$ and $X'_{\ell}$ are called \emph{siblings}.

On the other hand, let $X$ again be a closed 4-manifold, and consider a 2-sphere $S \subset X$ with $[S] \cdot [S] = 0$, so that $\overline{\eta(S)}$ is diffeomorphic to $S^2 \X D^2$.  Define $X_S = (X \setminus S) \cup (S^1 \X D^3)$.  We say that $X_S$ is the result of \emph{surgery on $S$}, and up to diffeomorphism, $X_S$ is unique~\cite{LP}.  If $X_{\ell}$ is obtained by surgery on a loop $\ell$ in $X$, then $X$ is obtained by surgery on the corresponding 2-sphere in $X_{\ell}$, and vice versa.

Now we will define the manifolds $S_p$ and $S_p'$ mentioned in the introduction.  Let $X = S^1 \X S^3$, and let $\ell_p$ be any loop in $S^1 \X S^3$ such that $[\ell_p] = p \in \Z = \pi_1(S^1 \X S^3)$.  (Note that $\ell_p$ is well-defined up to isotopy, since homotopy and isotopy coincide for loops in dimension four.)  Finally, define $S_p$ and $S_p'$ to be the two 4-manifolds obtained by surgery on $\ell_p$.  The manifold $S_p$ is also the spun lens space as described in~\cite{meier}, which contains a number of additional details about these manifolds.  In particular, Pao proved that $S_p$ and $S_p'$ are diffeomorphic if and only if $p$ is odd~\cite{pao}, and Meier produced $(3;1,1,1)$-trisections for $S_p$ and $S_p'$.  Meier's diagram for $S_2$ is shown in Figure~\ref{fig:spunrp2}.  Chu and Tillmann proved that for any closed 4-manifold $X$, we have $g(X) \geq \chi(X) - 2 + 3 \text{rk}(\pi_1(X))$\cite{CT}.  This inequality implies

\begin{lemma}
If $p \neq 1$, both $S_p$ and $S_p'$ have trisection genus three, and if $p \neq 0,1$, their genus-three trisections are irreducible.
\end{lemma}

\begin{proof}
The manifold $X = S_p$ or $X = S_p'$ satisfies $\chi(X) = 2$ and $\pi_1(X) = \Z/p\Z$.  Thus, if $p \neq 1$, we have $g(X) \geq 2 - 2 + 3$ by the inequality above.  By Theorem~\ref{thm:MZ}, if $X$ admits a reducible genus-three trisection, then $X$ can be expressed as a connected sum of $S^2 \X S^2$, $\pm \CP^2$, and $S^1 \X S^3$, and so if $p \neq 0,1$, it follows that any genus-three trisection of $X$ is irreducible.
\end{proof}
Note that when $p=0$, we have $S_0 = (S^2 \X S^2) \# (S^1 \X S^3)$ and $S_0' = \CP^2 \# \overline{\CP}^2 \# (S^1 \X S^3)$.  When $p = 1$, we have that $S_1$ and $S_1'$ are diffeomorphic to $S^4$.

\section{Technical lemmas}\label{sec:technical}

We begin to build up to our proof of the main theorem by importing three lemmas about genus-two Heegaard splittings from the work of Cho and Koda.

\begin{lemma}\cite[Lemma 2.5]{ChoKoda1}\label{lem:primitive}
Suppose $H \cup_{\Sigma} H'$ is a genus-two Heegaard splitting for $S^1 \X S^2$, with $c_1$ and $c_2$ disjoint curves bounding disks in $H$ such that $c_1$ and $c_2$ are primitive in $H'$.  Then there exists a compressing curve $c$ for $H$ such that $|c_1 \cap c| = |c_2 \cap c| = 1$.
\end{lemma}

\begin{lemma}\cite[Lemma 2.4]{ChoKoda1}\label{lem:unique1}
The genus-two Heegaard surface for $S^1 \X S^2$ has a unique non-separating reducing curve.
\end{lemma}

\begin{lemma}\cite[Lemma 1.9]{ChoKoda2}\label{lem:unique2}
The genus-two Heegaard surface for $L(p,q) \# S^1 \X S^2$ has a unique non-separating reducing curve.
\end{lemma}

The next five lemmas are required to proof Proposition~\ref{prop:dp}, which is the first major step on the way to proving Theorem~\ref{thm:main}.

\begin{lemma}\label{lem:g2thick}
If $\Sigma_i$ is a weakly reducible genus-two thick surface in a generalized Heegaard splitting for a 3-manifold $Y$, then $\Sigma_i$ is reducible.
\end{lemma}

\begin{proof}
Suppose $\Sigma_i$ is weakly reducible.  Then there exist disjoint compressing curves $c_1$ and $c_2$ in $\Sigma_i$ for $C_1^i$ and $C_2^i$, respectively.  Let $\Sigma_i'$ be the result of compressing $\Sigma_i$ along $c_1$ and $c_2$.  By a standard argument, since the genus of $\Sigma_i$ is two, at least one component of $\Sigma_i'$ is a sphere containing scars coming from both $c_1$ and from $c_2$.  Let $c$ be a curve separating in this component of $\Sigma_i'$ separating the scars coming from $c_1$ from the scars coming from $c_2$.  Then $c$ bounds a disk in $\Sigma_i'$ containing only scars coming from $c_1$, which implies that $c$ bounds a compressing disk in $C_1^i$.  Similarly, $c$ bounds a disk in $\Sigma_i'$ containing only scars coming from $c_2$, so $c$ bounds a compressing disk in $C_2^i$.  We conclude that $c$ is a reducing curve for $\Sigma_i$.
\end{proof}

\begin{lemma}\label{lem:comp1}
Suppose that $C$ is a compression-body such that $g(\pd_+ C) = 2$ and $g(\pd_- C) = 1$.  Then $C$ has a unique non-separating compressing disk up to isotopy.
\end{lemma}

\begin{proof}
Note that $C$ has at least one non-separating compressing disk $D$, and suppose by way of contradiction that $D'$ is another non-separating compressing disk which is not isotopic to $D$.  If $D \cap D' = \emp$, then $D' \subset C \setminus D$, where $C \setminus D$ is diffeomorphic to $T^2 \X I$.  As $\pd(T^2 \X I)$ is incompressible, it follows that $D'$ is isotopic to a disk $D^* \subset \pd (T^2 \X I)$.  Let $D_1$ and $D_2$ in $\pd(T^2 \X I)$ denote the two scars of cutting $C$ along $D$.  If $D^*$ contains neither scar, then $D'$ is boundary-parallel in $C$, violating our assumption that $\pd D$ is essential in $\pd_+ C$, and if $D^*$ contains both scars, then $D'$ is separating in $C$, again violating our hypotheses.  It follows that $D^*$ contains a single scar, which implies that $D$ and $D'$ are isotopic, a contradiction.

Now, suppose that $D \cap D' \neq \emp$, and suppose that $D'$ intersects $D$ minimally among all non-separating disks not isotopic to $D$.  Using a standard cut-and-paste argument and the irreducibility of $C$, we can assume that $D \cap D'$ contains only arcs of intersection.  Let $\A$ be an arc of intersection that is outermost in $D$, so that $\A$ co-bounds a disk $\Delta \subset D$ with an arc in $\pd_+ C$ and such that $\Delta$ does not intersect $D'$ in its interior.  Construct two new disks $D_1'$ and $D_2'$ by gluing a copy of $\Delta$ to each component of $D' \setminus \A$.  Then $D_1'$ and $D_2'$ are embedded disks with boundary in $\pd_+ C$, and since $[\pd D'] = [\pd D_1'] + [\pd D_2']$ in $H_1(\pd_+ C)$, at least one of $D_1'$ or $D_2'$ is a non-separating compressing disk.  Note that by construction, we have $D_1' \cap D_2' = \emp$, and $|D \cap D_i'| < |D \cap D'|$ for both $i=1$ and 2.

Suppose without loss of generality that $D_1'$ is non-separating.  Since $|D \cap D_1'| < |D \cap D'|$, we have by minimality of $D'$ that $D_1'$ and $D$ are isotopic.  Moreover, $[\pd D'] = [\pd D_1'] + [\pd D_2']$ in $H_1(\pd_+ C)$ implies that $[\pd D_2'] = 0$, so that $\pd D_2'$ is a separating curve.  If $\pd D_2'$ is inessential in $\pd_+ C$, then we can construct an isotopy from $D'$ to $D_1'$, implying $D'$ is isotopic to $D$ as well, a contradiction.  Otherwise, $\pd D_2'$ is separating, so that $C \setminus D_2'$ is the union of $T^2 \X I$ and a solid torus $V$ in which $D_1'$ is a compressing disk.  Now, we can reconstruct $D'$ by boundary-tubing $D_1'$ and $D_2'$ in $C$.  However, the result of such a boundary-tubing the separating curve $D_2'$ to the non-separating curve $D_1'$ in $C$ yields a disk isotopic to $D_1'$.  It follows that $D'$ is isotopic to $D_1'$ and thus to $D$, another contradiction.  We conclude that no such disk $D'$ exists.
\end{proof}

\begin{lemma}\label{lem:comp2}
Suppose that $C$ is a compression-body such that $g(\pd_+ C) = 2$ and $\pd_- C$ is a disjoint union of two tori.  Then $C$ has a unique compressing disk up to isotopy.
\end{lemma}

\begin{proof}
As in the proof of Lemma~\ref{lem:comp1}, we note that $C$ has one (separating) compressing disk $D$, and we suppose by way of contradiction that $D'$ is another compressing disk which is not isotopic to $D$.  If $D \cap D' = \emp$, then $D' \subset C \setminus D$, where $C \setminus D$ is diffeomorphic to two disjoint copies of $T^2 \X I$.  As $\pd(T^2 \X I)$ is incompressible, it follows that $D'$ is isotopic to a disk $D^* \subset \pd (T^2 \X I)$.  Let $D_1$ and $D_2$ denote the two scars of cutting $C$ along $D$, where one scar is in each copy of $\pd (T^2 \X I)$.  If $D^*$ contains neither scar, then $D'$ is boundary-parallel in $C$, violating our assumption that $\pd D$ is essential in $\pd_+ C$.  Otherwise, $D^*$ contains one of the scars, implying $D$ and $d'$ are isotopic, a contradiction.

Now, suppose that $D \cap D' \neq \emp$ with $D'$ intersecting $D$ minimally among all compressing disks not isotopic to $D$.  As above, we can assume that $D \cap D'$ contains only arcs of intersection, and if $\A$ is an arc of intersection that is outermost in $D$, we can use $\A$ to construct two new disks $D_1'$ and $D_2'$ such that $D_1' \cap D_2' = \emp$, and $|D \cap D_i'| < |D \cap D'|$ for both $i=1$ and 2.  If $D_1'$ is a boundary parallel disk, then we can use this parallelism to construct an isotopy from $D_2'$ to $D'$, contradicting our assumption that $D'$ intersects $D$ minimally.  Similarly, $D_2'$ cannot be boundary parallel, and thus the only possibility is that both $D_1'$ and $D_2'$ are isotopic to $D$ (and thus to each other).  However, as above, $D'$ is the result of such a boundary-tubing $D_1'$ and $D_2'$ in $C$, but boundary tubing two parallel separating disks produces a boundary-parallel disk, another contradiction.  We conclude that no such disk $D'$ exists.
\end{proof}

\begin{lemma}\label{lem:weak1}
Suppose $H_1 \cup_{\Sigma} H_2$ is a genus-three Heegaard splitting of $Y = S^3$ or $S^1 \X S^2$, and let $c_1$ and $c_2$ be a weak-reducing pair of mutually non-separating curves.  Then one of the following holds:
\begin{enumerate}
\item There exists a compressing curve $c_2'$ for $H_2$ such that $|c_2' \cap c_1| = 1$ and $c_2' \cap c_2 = \emp$,
\item There exists a compressing curve $c_1'$ for $H_1$ such that $|c_1' \cap c_2| = 1$ and $c_1' \cap c_1 = \emp$,
\item $c_1$ is a reducing curve, or
\item $c_2$ is a reducing curve.
\end{enumerate}
\end{lemma}

\begin{proof}
Untelescoping $\Sigma$ along the weak-reducing pair $c_1$ and $c_2$ yields a generalized Heegaard splitting $(\Sigma_{c_1}, \Sigma_{c_1,c_2}, \Sigma_{c_2})$, where
\[ Y = H'_1 \cup_{\Sigma_{c_1}} C_1 \cup_{\Sigma_{c_1,c_2}} C_2 \cup_{\Sigma_{c_2}} H'_2.\]
Let $D_i$ be a disk bounded by $c_i$ in $H_i$, so that $H_1' = H_1 \setminus D_1$ and $H_2' = H_2 \setminus D_2$.

Since $c_1$ and $c_2$ are mutually non-separating, the thin surface $\Sigma_{c_1,c_2}$ is a torus, and so it compresses in $Y$, as neither $S^3$ nor $S^1 \X S^2$ contains an incompressible torus.  By Proposition~\ref{prop:thin}, at least one of the surfaces $\Sigma_{c_1}$ or $\Sigma_{c_2}$ is weakly reducible.  Moreover, $g(\Sigma_{c_1}) = g(\Sigma_{c_2}) = 2$, and so by Lemma~\ref{lem:g2thick}, one of the thick surfaces is reducible.  Suppose first that $\Sigma_{c_1}$ is reducible, and let $c^*$ be a separating reducing curve.  Let $D^*$ denote the disk bounded by $c^*$ in $C_1$, so that $C_1 \setminus D^*$ is the union of a solid torus $V$ and a copy of $T^2 \X I$.  Now, both the meridian of $V$ and the curve $c_2$, viewed as a curve in $\Sigma_{c_1}$, bound a non-separating compressing disks for $C_1$.  By Lemma~\ref{lem:comp1}, such a disk is unique, and so it follows that $c_2$ is homotopic in $\Sigma_{c_1}$ to a curve disjoint from $c^*$.

Fixing $c_2$, there is a homotopy of $c^*$ in $\Sigma_{c_1}$ taking $c^*$ to a curve $c^{**}$ disjoint from $c_2$.  Following Remark~\ref{rmk:slide}, we have that there is a sequence of handleslides of $c^*$ over $c_1$ in our original surface $\Sigma$ yielding $c^{**}$.  See Figure~\ref{fig:weak}.  We note also that since $c^*$ bounds a disk in $H_1' = H_1 \setminus D_1$, we have $c^*$ (viewed as a curve in $\Sigma$) bounds a disk in $H_1$ as well; hence, the curve $c^{**}$ also bounds a disk in $H_1$.  Since $c^{**}$ is separating in $\Sigma_{c_1}$, it cuts off a once-punctured torus $T$ containing $c_2$, which a priori may contain one or both scars from $c_1$.  After additional slides of $c^{**}$ over these scars if necessary, we obtain a curve $c^{***}$ cutting off a punctured torus $T^*$ containing $c_2$ and no scars, so that $T^*$ is a subsurface of $\Sigma$ as well.  Since $c_2 \subset T^*$, it follows that $c^{***}$ bounds a disk in $H_2$, and thus $c^{***}$ is a reducing curve for $\Sigma$ cutting off a genus-one summand $Y^*$ of $Y$.  If $Y^* = S^3$, then conclusion (2) holds; otherwise, $Y^* = S^1 \X S^2$ and conclusion (4) holds.

The other possibility is that $\Sigma_{c_2}$ is reducible, in which case a parallel argument shows that (1) or (3) holds, completing the proof.
\end{proof}

\begin{figure}[h!]
  \centering
  \includegraphics[width=.4\linewidth]{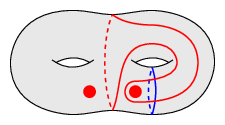} \qquad
    \includegraphics[width=.4\linewidth]{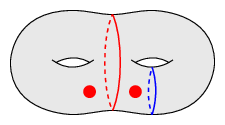}
    \put (-293,0) {\textcolor{red}{$c^*$}}
\put (-260,0) {\textcolor{blue}{$c_2$}}
    \put (-93,0) {\textcolor{red}{$c^{**}$}}
\put (-60,0) {\textcolor{blue}{$c_2$}}\\
  \includegraphics[width=.4\linewidth]{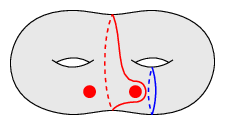}
      \put (-93,0) {\textcolor{red}{$c^{***}$}}
\put (-60,0) {\textcolor{blue}{$c_2$}}
  \caption{At top left, curves $c^*$ and $c_2$ are homotopic to disjoint curves in $\Sigma_{c_1}$ (shown with scars).  At top right, slides of $c^*$ over $c_1$ in $\Sigma$ (homotopies over scars in $\Sigma_{c_1}$) yield $c^{**}$ disjoint from $c_2$ in $\Sigma_{c_1}$.  At bottom, further slides of $c^{**}$ over $c_1$ yield a reducing curve $c^{***}$ for $\Sigma$.}
	\label{fig:weak}
\end{figure}
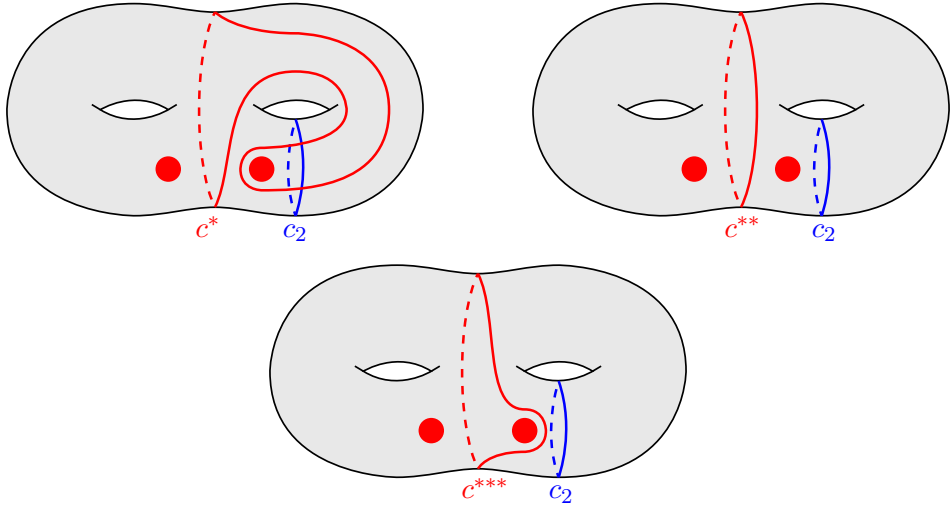

\begin{lemma}\label{lem:weak2}
Suppose $H_1 \cup_{\Sigma} H_2$ is a genus-three Heegaard splitting of $Y = S^3$ or $S^1 \X S^2$, and let $c_1$ and $c_2$ be a weak-reducing pair of non-separating but mutually separating curves.  Then either $c_1$ is a reducing curve or $c_2$ is a reducing curve.
\end{lemma}

\begin{proof}
We proceed as in the proof of Lemma~\ref{lem:weak1}, untelescoping to get a generalized Heegaard splitting $(\Sigma_{c_1}, \Sigma_{c_1,c_2}, \Sigma_{c_2})$, where
\[ Y = H'_1 \cup_{\Sigma_{c_1}} C_1 \cup_{\Sigma_{c_1,c_2}} C_2 \cup_{\Sigma_{c_2}} H'_2.\]
As above, we let $c_i$ bounds a disk $D_i$ in $H_i$, so that $H_i' = H_i \setminus D_i$.

In this case, the thin surface $\Sigma_{c_1,c_2}$ is a disjoint union of two tori, and so it compresses in $Y$.  By Proposition~\ref{prop:thin} and Lemma~\ref{lem:g2thick}, at least one of the thick surfaces is reducible.  Suppose that $\Sigma_{c_1}$ has a separating reducing curve $c^*$.  Then both $c^*$ and $c_2$ bound separating compressing disks in the compression-body $C_1$, and so Lemma~\ref{lem:comp2}, the uniqueness of such a disk implies that $c_2$ and $c^*$ are homotopic in $\Sigma_{c_1}$.  As in the proof of Lemma~\ref{lem:weak1}, it follows that there is a sequence of handleslides of $c^*$ over $c_1$ in $\Sigma$ yielding $c_2$.  Since $c^*$ bounds a disk in $H_1'$, it also bounds a disk in $H_1$, and since and $c_1$ bounds a disk in $H_1$, so does $c_2$.  We conclude that $c_2$ is a reducing curve for $\Sigma$.  A parallel argument shows that if $\Sigma_{c_2}$ is reducible, then $c_1$ is a reducing curve, competing the proof.
\end{proof}

As noted earlier, the preceding lemmas comprise the key ingredients in the proof of the next proposition.

\begin{proposition}\label{prop:dp}
Suppose $\T$ is a weakly reducible genus-three trisection, with a trisection diagram $(\Sigma;\A,\n,\g)$ such that $\A_1$ disjoint from $\n_3 = \g_3$.  Then either $\T$ is reducible, or there exist curves $\n_1$ and $\g_1$ bounding disks in $H_{\n}$ and $H_{\g}$, respectively, disjoint from $\n_3$, and such that $|\A_1 \cap \n_1| = 1$ and $|\A_1 \cap \g_1| = 1$.
\end{proposition}

\begin{proof}
As noted in the previous section, we assume $k_i \leq 1$ for all $i$.  First, suppose that $\A_1$ and $\n_3$ are mutually separating.  Considering the Heegaard splitting $H_{\A} \cup_{\Sigma} H_{\n}$ of $Y = S^3$ or $S^1 \X S^2$.  By Lemma~\ref{lem:weak2}, either $\A_1$ bounds a disk in $H_{\n}$ or $\n_3$ bounds a disk in $H_{\A}$.  On the other hand, considering the Heegaard splitting $H_{\A} \cup_{\Sigma} H_{\n}$ of $Y = S^3$ or $S^1 \X S^2$, we have that either $\A_1$ bounds a disk in $H_{\g}$ or $\n_3$ bounds a disk in $H_{\A}$.  If $\n_3$ bounds a disk in $H_{\A}$, then $\T$ is reducible.  If $\n_3$ does not bound a disk in $H_{\A}$, the only possibility is that $\A_1$ bounds disks in $H_{\n}$ and $H_{\g}$, and again $\T$ is reducible.

Next, suppose that $\A_1$ and $\n_3$ are mutually non-separating.  Applying Lemma~\ref{lem:weak1} to $H_{\A} \cup_{\Sigma} H_{\n}$ yields that $\A_1$ is primitive in $H_{\n}$, $\A_1$ bounds a disk in $H_{\n}$, $\n_3$ is primitive in $H_{\A}$, or $\n_3$ bounds a disk in $H_{\A}$.  Both of the latter two possibilities together with the assumption that $\n_3$ bounds disks in both $H_{\n}$ and $H_{\g}$ imply that $\T$ is reducible, because if $\n_3$ is primitive in $H_{\A}$, then $\T$ is stabilized.  Similarly, applying Lemma~\ref{lem:weak1} to $H_{\A} \cup_{\Sigma} H_{\g}$ yields that $\A_1$ is primitive in $H_{\g}$, $\A_1$ bounds a disk in $H_{\g}$, $\g_3$ is primitive in $H_{\A}$, or $\g_3$ bounds a disk in $H_{\A}$.  As before, the second two cases imply that $\T$ is reducible.

Thus, suppose that $\A_1$ is primitive in $H_{\n}$ or $\A_1$ bounds a disk in $H_{\n}$, and in addition, suppose that $\A_1$ is primitive in $H_{\g}$ or $\A_1$ bounds a disk in $H_{\g}$.  If $\A_1$ bounds a disk in $H_{\n}$, then either $\A_1$ is primitive in $H_{\g}$, so $\T$ is stabilized, or $\A_1$ bounds a disk in $H_{\g}$, so $\A_1$ is a reducing curve.  A parallel argument show that if $\A_1$ bounds a disk in $H_{\g}$, then $\T$ must be reducible.  The only remaining unaddressed pair of cases is that $\A_1$ is primitive in $H_{\n}$ and $\A_1$ is primitive in $H_{\g}$.  By Lemma~\ref{lem:weak1}, there must be curves $\n_1$ and $\g_1$ bounding disks in $H_{\n}$ and $H_{\g}$, respectively, disjoint from $\n_3$, and such that each of $\n_1$ and $\g_1$ intersects $\A_1$ in a single point.
\end{proof}

\section{Heegaard triples}\label{sec:triple}

In the course of proving Theorem~\ref{thm:main}, we use the idea of a \emph{Heegaard triple}, which is closely related to a trisection diagram.  A \emph{Heegaard triple} is a tuple $(\Sigma;\A,\n,\g)$ that consists of a surface $\Sigma$ and a collection of three cut systems, $\A$, $\n$, and $\g$.  We let $Y_1 = H_{\g} \cup H_{\A}$, $Y_{2} = H_{\A} \cup H_{\n}$, and $Y_{3} = H_{\n} \cup H_{\g}$, so that pairings of cut systems yield Heegaard diagrams for the 3-manifolds $Y_i$.  (The indexing may seem counterintuitive, but these choices are consistent with our conventions for trisection diagrams, in which $Y_i = \pd X_i$).  Two Heegaard triples are said to be equivalent if the corresponding cut systems are related by a finite sequence of handleslides in $\Sigma$.  As such, two Heegaard triples are equivalent if and only if they determine the same three handlebodies $H_{\A}$, $H_{\n}$, and $H_\g$~\cite{johan}.

A trisection diagram is then a special case of a Heegaard triple in which $Y_i = \#^{k_i} (S^1 \X S^2)$.  In this section, we examine Heegaard triples which are *not* trisection diagrams, but which will arise naturally in our later proof.  More precisely, we study Heegaard triplets $(\Sigma;\alpha,\beta,\gamma)$ where one pair, say $(\Sigma; \alpha, \gamma)$, is a Heegaard diagram for $\#^k(S^1\times S^2)$. These diagrams are discussed in \cite{OZ_hol} and \cite{thin_4m}. Recall the definition of a \emph{standard diagram} from Section~\ref{sec:prelim}.

\begin{lemma}\label{lem:handle}
Suppose $(\Sigma;\A,\n,\g)$ is a genus-$g$ Heegaard triple such that $(\Sigma;\A,\g)$ is a standard Heegaard diagram for $\#^k(S^1 \X S^2)$.  Then $Y_{3}$ is obtained by Dehn surgery on the $g-k$-component link in $Y_{2}$ consisting of the $g-k$ curves of $\g - \A$ with framing determined by $\Sigma$.
\end{lemma}

\begin{proof}
The proof of this lemma is an adaptation of the proof of Lemma 13 of~\cite{GK} to the present setting.  If $(\Sigma;\A,\g)$ is a Heegaard diagram for $\#^k(S^1 \X S^2)$, we can build a 4-dimensional cobordism $X$ from $Y_{2}$ to $Y_{3}$ by first taking a 4-dimensional regular neighborhood of the central surface $\Sigma$ and gluing in copies of $H_{\A} \X I$, $H_{\n} \X I$, and $H_{\g} \X I$ along the corresponding cut systems.  The resulting 4-manifold has three boundary components, $Y_{1}$, $Y_{2}$, and $Y_{3}$, and we cap off $Y_{1} = \#^k(S^1 \X S^2)$ with $\natural^k (S^1 \X D^3)$, yielding $X$.

The proof of Lemma 13 of~\cite{GK} shows that a relative handle decomposition for $X$ is obtained by starting with $Y_{2} \X I$ and then attaching $g-k$ 4-dimensional 2-handles along the curves in $\g - \A$ with framing determined by the surface $\Sigma$.  The statement of the present lemma follows immediately.
\end{proof}

We can then use Lemma~\ref{lem:handle} to show that certain Heegaard triples do not exist.  We let $L(p,q)$ refer to a lens space that is not $S^3$ or $S^1 \X S^2$, so that $p,q \geq 2$.

\begin{lemma}\label{lem:tripexist}
There does not exist a genus-two Heegaard triple $(\Sigma;\A,\n,\g)$ such that $Y_{2} = S^3$, $Y_{3} = S^1 \X S^2$, and $Y_{1} = L(p,q) \# (S^1 \X S^2)$ or $Y_{1} = (S^1 \X S^2) \# (S^1 \X S^2)$.
\end{lemma}

\begin{proof}
Suppose by way of contradiction that such $(\Sigma;\A,\n,\g)$ exists.  By Lemma~\ref{lem:handle}, it follows that $L(p,q) \# (S^1 \X S^2)$ is obtained by surgery on a knot in $S^3$.  However, if $Y$ is any 3-manifold obtained in this way, $H_1(Y)$ is cyclic, while $H_1(L(p,q) \# (S^1 \X S^2))$ is $\Z_p \oplus \Z$ and $H_1((S^1 \X S^2) \# (S^1 \X S^2))$ is $\Z^2$, a contradiction.
\end{proof}

In the next proposition, we adapt a proof from~\cite{MSZ} to a different class of Heegaard triples.

\begin{proposition}\label{prop:lens}
Suppose $(\Sigma;\A,\n,\g)$ is a genus-two Heegaard triple such that $Y_{2} = Y_{3} = S^1 \X S^2$ and $Y_{1} = L(p,q) \# (S^1 \X S^2)$.  Then there exists a non-separating curve $c \subset \Sigma$ that bounds disks in each of $H_{\A}$, $H_{\n}$, and $H_{\g}$.
\end{proposition}

\begin{proof}
Possibly after performing some handle-slides, we may suppose that $(\Sigma;\A,\n,\g)$ is a Heegaard triple such that the Heegaard diagram $(\Sigma;\n,\g)$ is a standard diagram for $S^1 \X S^2$, with $\n_1 = \g_1$ and $|\n_2 \cap \g_2| = 1$.  By Lemma~\ref{lem:handle}, we have that $L(p,q) \# (S^1 \X S^2)$  is obtained by Dehn surgery on $\g_2$, viewed as a knot in $Y_2 = S^1 \X S^2$.  Since both $S^1 \X S^2$ and $L(p,q) \# (S^1 \X S^2)$ are reducible, and since $S^1 \X S^2$ contains a sphere not bounding a rational homology ball, the main theorem from~\cite{Schar} implies that either $(S^1 \X S^2) \setminus \g_2$ is reducible, or $\g_2$ is cabled in $S^1 \X S^2$ and the surgery slope is the slope of the cabling annulus.

We claim that $\g_2$ is contained in a 3-ball in $S^1 \X S^2$.  If $(S^1 \X S^2) \setminus \g_2$ is reducible, we note that every separating 2-sphere in $S^1 \X S^2$ bounds a 3-ball on one side, and so $\g_2$ must be contained in the 3-ball bounded by the reducing sphere.  On the other hand, if $\g_2$ is cabled with companion $J$, surgery on $\g_2$ yields $L \# Y$, where $L$ is some lens space and $Y$ is the result of surgery on $J$.  Since $L \# Y = L(p,q) \# (S^1 \X S^2)$, it follows that $J$ is a knot in $S^1 \X S^2$ with a non-trivial $S^1 \X S^2$ surgery, and so by~\cite{Gabai1}, we have that $(S^1 \X S^2) \setminus J$ is reducible, implying that $J$, and thus $\g_2$, is contained in a 3-ball (and consequently, $(S^1 \X S^2) \setminus \g_2$ is reducible in this case as well).

Pushing $\g_2$ into $H_{\n}$, we have that $H_\A \cup_{\Sigma} (H_\n \setminus \g_2)$ is a Heegaard splitting for $(S^1 \X S^2) \setminus \g_2$, and by Haken's Lemma~\ref{lem:haken}, $\Sigma$ is reducible, considered as a Heegaard surface for this splitting.  Since $(S^1 \X S^2) \setminus \g_2 = (S^1 \X S^2) \# (S^3 \setminus \g_2)$ and $g(\Sigma) = 2$, the reducing sphere for $\Sigma$ cuts off a genus-one splitting of $S^1 \X S^2$, and so there exists a non-separating curve $c$ in $\Sigma$ bounding disks in both $H_\A$ and $H_\n \setminus \g_2$.  By construction, the curve $\n_1$ also bounds a non-separating disk in the compression-body $H_\n \setminus \g_2$.  Thus, Lemma~\ref{lem:comp1} implies that $c$ and $\n_1$ coincide, so that $\n_1$ also bounds a disk in $H_\A$.  But we assumed above that $\n_1 =\g_1$ as well, and we conclude that $c$ bounds a disk in each of $H_\A$, $H_\n$, and $H_\g$.
\end{proof}

We require a second and significantly more complicated classification result for genus-two Heegaard triples, in which the three 3-manifolds given by pairing cut systems are $S^3$, $S^3$, and $L(p,q) \# (S^1 \X S^2)$.  This proof follows the road map of the main result of~\cite{MZ2}, that $(2;0)$-trisections are standard.  The starting point for this theorem is that every non-standard genus-two Heegaard diagram for $S^3$ contains a \emph{wave}, an arc $\nu$ with endpoints the same side of one curve in the diagram and which avoids the other curves~\cite{HOT}.  If a diagram admits a wave, a process called \emph{wave surgery} produces a handleslide that reduces the total number of intersections in the diagram.  See Figure~\ref{fig:wave} for an example, and for further details, see Section 2 of~\cite{MZ2}.

\begin{figure}[h!]
  \centering
  \includegraphics[width=.4\linewidth]{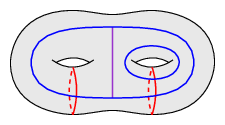} \qquad
    \includegraphics[width=.4\linewidth]{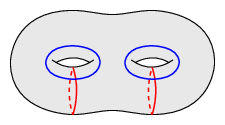}

  \caption{At left, a non-standard genus-two diagram for $S^3$ and a wave, surgery on which yields the standard genus-two diagram at right.}
	\label{fig:wave}
\end{figure}

Suppose $Y$ is a genus-two 3-manifold that can be expressed as $Y = Y_1 \# Y_2$.  Recall that a Heegaard diagram $(\Sigma;\A,\n)$ for $Y$ is \emph{standard} if there exists a separating curve $c$ disjoint from $\A \cup \n$ and \emph{non-standard} otherwise.  Such a 3-manifold $Y$ is said to have the \emph{wave property} if every non-standard genus-two Heegaard diagram admits a wave.  Relevant to our work here, Negami and Okita proved
\begin{theorem}\cite{NO}\label{thm:NO}
$L(p,q) \# (S^1 \X S^2)$ has the wave property.
\end{theorem}
As in~\cite{MZ2}, we define the \emph{intersection matrix} $M(\A,\n)$ for a genus-two Heegaard diagram $(\Sigma;\A,\n)$ to be
\[ M(\A,\n) = \begin{bmatrix} \Delta(\A_1, \n_1) & \Delta(\A_1, \n_2) \\ \Delta(\A_2, \n_1) & \Delta(\A_2, \n_2) \end{bmatrix},\]
where $\Delta(\A_i, \n_j)$ represents the algebraic intersection number of the two curves.
\begin{lemma}\label{lem:det}
If $(\Sigma;\A,\n)$ is a genus-two Heegaard diagram for $L(p,q) \# (S^1 \X S^2)$, then $\det(M(\A,\n)) = 0$.
\end{lemma}

\begin{proof}
First, note that the standard diagram $(\Sigma;\A,\n)$ yields
\[ M(\A,\n) = \begin{bmatrix} p & 0 \\ 0 & 0 \end{bmatrix},\]
and so the lemma is true for the standard diagram.  By Haken's Lemma~\ref{lem:haken}, every genus-two Heegaard splitting of $L(p,q) \# (S^1 \X S^2)$ is a connected sum of standard genus-one splittings, and so every genus-two Heegaard diagram of $L(p,q) \# (S^1 \X S^2)$ is handleslide equivalent to the standard one.  Finally, handleslides preserve $|\det(M(\A,\n))|$, completing the proof.
\end{proof}

As part of our classification, we will need

\begin{proposition}\cite[Proposition 3.1]{MZ2}\label{prop:trip}
Suppose $(\Sigma;\A,\n,\g)$ is a genus-two Heegaard triple such that $Y_{2} = S^3$, both $Y_{1}$ and $Y_{3}$ have the wave property, and both $|\det(M(\A,\g))|$ and $|\det(M(\n,\g))|$ are at most one.  Then there exists an equivalent Heegaard triple $(\Sigma;\A',\n',\g')$ such that $(\Sigma;\A',\n')$ is a standard and either $(\Sigma;\A',\g')$ or $(\Sigma;\n',\g')$ is standard.
\end{proposition}

\begin{remark}
In~\cite{MZ2}, Proposition 3.1 is stated for $(2;0)$-trisections, so that $Y_{1} = Y_{3} = S^3$.  However, its proof (which spans Sections 3, 4, and 5 of that paper and is over 20 pages long), uses only two properties of $Y_{1}$ and $Y_{3}$, that they have the wave property (used throughout the argument) and that $|\det(M(\A,\g))|$ and $|\det(M(\n,\g))|$ are at most one (used in the proof of Lemma 3.5 of \cite{MZ2}).  Thus, while we have stated the conclusion of Proposition 3.1 of \cite{MZ2} here in the greater generality, the proof remains the same.
\end{remark}

The remainder of the section is dedicated to proving the next theorem, which, in essence, shows that a certain type of genus-two Heegaard triple is \emph{stabilized}, in the sense that there are dual curves bounding disks in the three handlebodies $H_{\A}$, $H_{\n}$, and $H_{\g}$, to be used later as an ingredient in the proof of Theorem~\ref{thm:main}.

\begin{theorem}\label{thm:triple}
Suppose $(\Sigma;\A,\n,\g)$ is a genus-two Heegaard triple such that $Y_{1} = Y_{2} = S^3$ and $Y_{3} = L(p,q) \# (S^1 \X S^2)$.  Then there exist curves $c_1$ bounding a disk in $H_{\A}$ and $c_2$ bounding disks in both $H_{\n}$ and $H_{\g}$ such that $|c_1 \cap c_2| = 1$.
\end{theorem}

Suppose that $(\Sigma;\A,\n,\g)$ satisfies the hypotheses of Theorem~\ref{thm:triple}.  Then we can apply Theorem~\ref{thm:NO}, Lemma~\ref{lem:det}, and Proposition~\ref{prop:trip} to suppose without loss of generality that $(\Sigma;\A,\n)$ is standard and either $(\Sigma;\A,\g)$ or $(\Sigma; \n,\g)$ is standard.  In the latter case, $\n$ and $\g$ have a curve, say $\n_1$ in common, and in addition, $|\n_1 \cap \A_1| = 1$, and the theorem holds.  The more difficult case is the former, and so we will suppose that $(\Sigma;\A,\g)$ is also a standard diagram for $S^3$.

We let $\iota(c,c')$ denote the geometric intersection of two curves and, as above, $\Delta(c,c')$ the algebraic intersection.  We orient $\A$, $\n$, and $\g$ so that $\Delta(\A,\n) = 2$ and $\Delta(\A,\g) = 2$.  As in~\cite{MZ2}, we obtain the \emph{Whitehead graph} $\Sigma_{\A}(\n,\g)$ by cutting $\Sigma$ along $\alpha$ to obtain $\Sigma_{\A}$, a sphere with four boundary components, denoted $\A_1^{\pm}$ and $\A_2^{\pm}$.  In $\Sigma_{\A}$, each of $\n$ and $\g$ become a pair of essential arcs.  Essential arcs in $\Sigma_{\A}$ can be parametrized by the extended rational numbers $\Q \cup \{\frac{1}{0}\}$, which we call the \emph{slope} of an essential arc, noting that both arcs of $\n$ have the same slope in $\Sigma_{\A}$, as do both arcs of $\g$.  As in Lemma 6.1 of~\cite{MZ2}, we normalize this parametrization so that arcs of $\g$ have slope $\frac{1}{0}$, and letting $\frac{m}{n}$ denote the slope of the $\n$ arcs, we may suppose without loss of generality that $-\frac{1}{2} < \frac{m}{n} \leq \frac{1}{2}$.  We note that $m$ must be odd and $n$ must be even, since the $\beta$ arcs connect $\A_1^-$ to $\A_1^+$ and $\A_2^-$ to $\A_2^+$.

Following,~\cite{MZ2}, we distinguish between two different types of points of intersection of $\n$ and $\g$.  A point $x \in \n \cap \g$ is \emph{inessential} if there is a homotopy of $\n \cup \g$ pushing $x$ into $\A_i^{\pm}$ (without increasing the number of intersections of $\n$ or $\g$ with $\A_i$) and \emph{essential} otherwise.  The \emph{winding} $W_i$ of $\n_i$ relative to $\g_i$ at $\A_i$ is defined to be the number of inessential points of intersection homotopic into $\A_i$, counted with sign.  See Figure~\ref{fig:graphex} for an example.

\begin{figure}[h!]
  \centering
  \includegraphics[width=.4\linewidth]{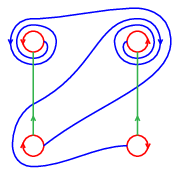}
\put (-147,28) {\textcolor{red}{$\A_1^-$}}
\put (-147,128) {\textcolor{red}{$\A_1^+$}}
\put (-47,28) {\textcolor{red}{$\A_2^-$}}
\put (-47,128) {\textcolor{red}{$\A_2^+$}}
\put (-154,88) {\textcolor{ForestGreen}{$\g_1$}}
\put (-37,68) {\textcolor{ForestGreen}{$\g_2$}}
\put (-93,67) {\textcolor{blue}{$\n_1$}}
\put (-113,105) {\textcolor{blue}{$\n_2$}}
  \caption{An example of a Whitehead graph $\Sigma_{\A}(\n,\g)$, in which $\frac{m}{n} = \frac{1}{2}$, $W_1 = 2$, and $W_2 = -2$.}
	\label{fig:graphex}
\end{figure}
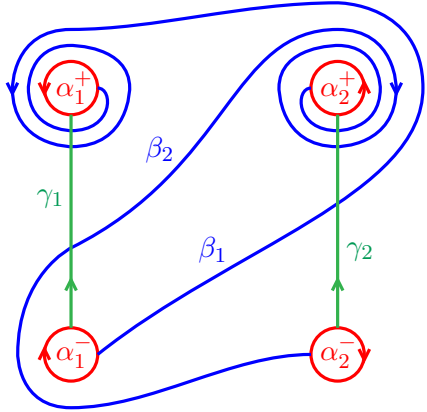

We import two results from~\cite{MZ2}, restating them slightly so that they apply to the context here.

\begin{proposition}\label{prop:restrict}
If $(\Sigma;\n,\g)$ contains a wave, then $|m| \leq 1$ and $|n| \leq 2$.
\end{proposition}

\begin{proof}
Lemma 6.2 from~\cite{MZ2} shows that if $|m| > 1$, then $(\Sigma;\n,\g)$ does not contain a wave.  Similarly, the second paragraph of the proof of Theorem 6.5 of~\cite{MZ2} demonstrates that if $|n| \geq 4$, then $(\Sigma;\n,\g)$ does not contain a wave.  The statement of the proposition follows from the fact that $n$ is even.
\end{proof}

\begin{proof}[Proof of Theorem~\ref{thm:triple}]
Suppose $(\Sigma;\A,\n,\g)$ is a genus-two Heegaard triple such that $Y_{1} = Y_{2} = S^3$ and $Y_{3} = L(p,q) \# (S^1 \X S^2)$.  By Theorem~\ref{thm:NO}, Lemma~\ref{lem:det}, and Proposition~\ref{prop:trip}, we may assume that $(\Sigma;\A,\n)$ is standard and either $(\Sigma;\A,\g)$ or $(\Sigma; \n,\g)$ is standard.  As noted above, the latter implies the theorem holds, and so we suppose that $(\Sigma;\A,\g)$ is standard.

By Proposition~\ref{prop:restrict}, and after applying a symmetry to $\Sigma$ if necessary, we may assume that $\frac{m}{n} = \frac{1}{0}$ or $\frac{1}{2}$.  If $\frac{m}{n} = \frac{1}{0}$, then there is a separating curve $c^*$ disjoint from $\A$, $\n$, and $\g$, implying that $(\Sigma;\n,\g)$ is standard and completing the proof as above.  See Figure~\ref{fig:n0}.  Otherwise, we suppose that $\frac{m}{n} = \frac{1}{2}$.  In this case, we compute
\[ M(\n,\g) = \begin{bmatrix}
W_1 & 1 \\ 1 & W_1 \end{bmatrix},\]
and by Lemma~\ref{lem:det}, we have $0 = \det(M(\n,\g)) = W_1W_2 - 1$.  It follows that $W_1 = W_2 = \pm 1$.  In either case, we can find a curve $c_2$ disjoint from both $\n$ and $\g$ and meeting both $\A_1$ and $\A_2$ in one point.  See Figure~\ref{fig:n2}.  Thus, $c_2$ bounds a disk in $H_{\n}$ and $H_{\g}$, and if we let $c_1 = \A_1$, the proof of the theorem is complete.
\end{proof}

\begin{figure}[h!]
  \centering
  \includegraphics[width=.4\linewidth]{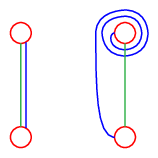}
  \put (-157,21) {\textcolor{red}{$\A_1^-$}}
\put (-157,134) {\textcolor{red}{$\A_1^+$}}
\put (-43,21) {\textcolor{red}{$\A_2^-$}}
\put (-43,134) {\textcolor{red}{$\A_2^+$}}
\put (-163,88) {\textcolor{ForestGreen}{$\g_1$}}
\put (-48,88) {\textcolor{ForestGreen}{$\g_2$}}
\put (-142,88) {\textcolor{blue}{$\n_1$}}
\put (-81,88) {\textcolor{blue}{$\n_2$}}
  \caption{The case $\frac{m}{n} = \frac{1}{0}$}
	\label{fig:n0}
\end{figure}

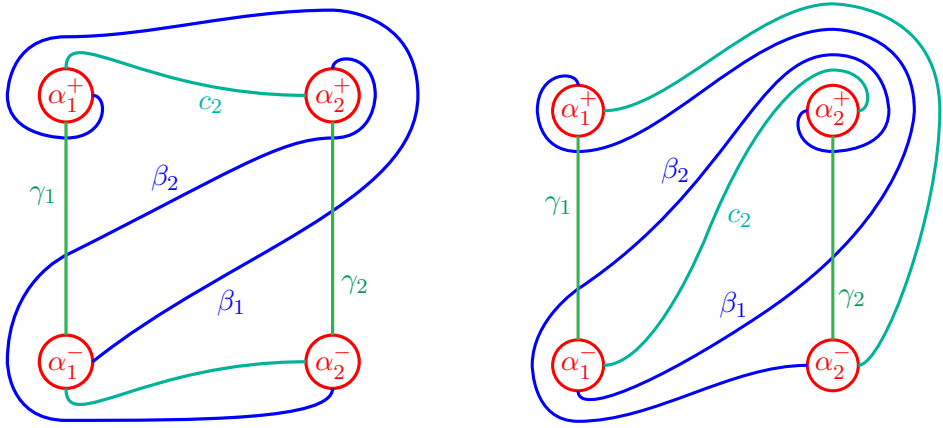
\begin{figure}[h!]
  \centering
  \includegraphics[width=.4\linewidth]{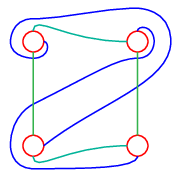} \qquad
  \includegraphics[width=.4\linewidth]{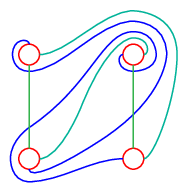}  
    \put (-347,27) {\textcolor{red}{$\A_1^-$}}
\put (-347,128) {\textcolor{red}{$\A_1^+$}}
\put (-246,27) {\textcolor{red}{$\A_2^-$}}
\put (-246,128) {\textcolor{red}{$\A_2^+$}}
\put (-354,92) {\textcolor{ForestGreen}{$\g_1$}}
\put (-236,58) {\textcolor{ForestGreen}{$\g_2$}}
\put (-283,50) {\textcolor{blue}{$\n_1$}}
\put (-308,97) {\textcolor{blue}{$\n_2$}}
\put (-290,126) {\textcolor{BlueGreen}{$c_2$}}
  \put (-153,27) {\textcolor{red}{$\A_1^-$}}
\put (-153,123) {\textcolor{red}{$\A_1^+$}}
\put (-57,27) {\textcolor{red}{$\A_2^-$}}
\put (-57,123) {\textcolor{red}{$\A_2^+$}}
\put (-159,88) {\textcolor{ForestGreen}{$\g_1$}}
\put (-48,53) {\textcolor{ForestGreen}{$\g_2$}}
\put (-93,49) {\textcolor{blue}{$\n_1$}}
\put (-115,100) {\textcolor{blue}{$\n_2$}}
\put (-90,83) {\textcolor{BlueGreen}{$c_2$}}
  \caption{The case $\frac{m}{n} = \frac{1}{2}$ and $W_1 = W_2 = 1$ (left) or $W_1 = W_2 = -1$ (right).  In either instance, we find a teal curve $c_2$ disjoint from $\n$ and $\g$ and meeting each curve in $\A$ once.}
	\label{fig:n2}
\end{figure}

As an aside, we state a corollary that may be of independent interest.  A 2-component link $L$ in $S^3$ is \emph{tunnel number one} if there exists an embedded arc $\tau$ which meets $L$ only in its endpoints and such that $S^3 \setminus (L \cup \tau)$ is a genus-two handlebody.  In this case, $L$ is isotopic into a genus-two Heegaard surface for $S^3$ with any possible integral framing.  So that we do not take too much of a detour, we sketch the argument here, but it parallels the argument which uses the main theorem from~\cite{MZ2} to prove Corollary 1.4 in that paper; see also~\cite{MZDehn} for further connections between trisections and Dehn surgeries.

\begin{corollary}\label{cor:surgery}
Suppose $L$ is a tunnel number one link in $S^3$ with an integral surgery to $L(p,q) \# (S^1 \X S^2)$.  Then $L$ is handle-slide equivalent to an unlink.
\end{corollary}

\begin{proof}
Suppose $L$ is a tunnel number one link in $S^3$ with the requisite surgery, and let $S^3 = H_1 \cup_{\Sigma} H_2$ be the genus-two Heegaard splitting for $S^3$.  Then there is an embedding of $L$ in $\Sigma$ such that the surface framing of $L$ agrees with the framing of the surgery, and $L$ is dual to a cut system $\A$ for $H_1$.  Letting $\n$ be a cut system for $H_2$ and letting $\g = L$, we have that $(\Sigma;\A,\n,\g)$ is a Heegaard triple in which $Y_1 = S^3$ and $Y_2 = S^3$.  Moreover, we can apply Lemma~\ref{lem:handle}, which asserts that $Y_3$ is the result of surface-framed surgery on $L$ in $H_1 \cup_{\Sigma} H_2$; that is, $Y_3 = L(p,q) \# (S^1 \X S^2)$.  By Theorem~\ref{thm:triple}, $(\Sigma;\A,\n,\g)$ is equivalent to a diagram $(\Sigma;\A',\n',\g')$ such that $\n_1' = \g_1'$ and $|\A_1' \cap \n_1'|=1$.  By sliding the remaining curves if necessary, the resulting diagram appears as in Figure~\ref{fig:n0}.  In this case, the curves $\g'$, viewed a link in $Y_2$, comprise the split union of the 0-framed unknot $\g_1'$ and the $q$-framed unknot $\g_2'$, which is a $(1,q)$-torus knot in $Y_2$.  Note that $\g = L$ and $\g'$ are related by handleslides in $\Sigma$, completing the proof.
\end{proof}

\section{Five-chains and surgery}\label{sec:fivechain}

In this section, we introduce five-chains and connect five-chains with the surgery operations described in Section~\ref{sec:prelim}.  In the next section, we will prove that a weakly reducible genus-three trisection is reducible or contains a five-chain.  A five-chain is shown at left in Figure~\ref{fig:fivesurgery}, and the result of surgery on the five-chain, a new trisection diagram in which the genus has been reduced by one, is shown at right.  Below, we invoke the work of the first author and Moeller on $\star$-trisection diagrams in \cite{AM} to prove Proposition~\ref{prop:surgery}, that if a trisection $\T'$ of $X'$ is the result of five-chain surgery on a trisection $\T$ of $X$, then $X$ is obtained by surgery on a loop in $X'$.  

\begin{figure}[h!]
  \centering
  \includegraphics[width=.25\linewidth]{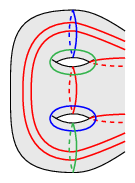} \quad \raisebox{2.5cm}{$\longrightarrow$}
    \includegraphics[width=.25\linewidth]{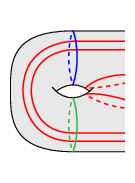}

  \caption{Obtaining $(\Sigma_{\A_1};\A',\n',\g')$ (right) by five-chain surgery on $(\Sigma;\A,\n,\g)$ (left).}
	\label{fig:fivesurgery}
\end{figure}

We begin by describing the $\star$-trisection machinery of~\cite{AM}:  Suppose $X$ admits a genus-$g$ trisection $\T$ with diagram $(\Sigma;\A,\n,\g)$.  A \emph{decomposed curve} $\ell$ is an immersed curve $\ell$ in $\Sigma$ which is the endpoint union of three embedded arcs, $a$, $b$, and $c$, such that $a \cap \A = b \cap \n = c \cap \g = \emp$.  (In this paper, all decomposed curves will be embedded but the theory works more generally.)  Define $\Sigma'$ to be the genus-$(g+2)$ surface obtained by removing disk neighborhoods of the boundary points of the three arcs $a$, $b$, and $c$ and attaching a pair of pants $P$.  A properly embedded arc in $P$ is called a \emph{wave} if both of its endpoints are contained in the same component of $\pd P$ and a \emph{seam} if both endpoints are in different components of $\pd P$.  Let $a'$, $b'$, and $c'$ be pairwise disjoint and non-isotopic seams in $P$ such that $\A_0 = a \cup a'$, $\n_0 = b \cup b'$, and $\g_0 = c \cup c'$ are simple closed curves in $\Sigma'$.  Finally, let $\A_0'$ be component of $\pd P$ that meets $b'$ and $c'$, let $\n_0'$ be the component of $\pd P$ that meets $c'$ and $a'$, and let $\g_0'$ be the component of $\pd P$ that meets $a'$ and $b'$.

Now, define $\A' = \A \cup \{\A_0,\A_0'\}$, $\n' = \n \cup \{\n_0, \n_0'\}$, and $\g' = \g \cup \{\g_0, \g_0'\}$.  We say that $(\Sigma';\A',\n',\g')$ is obtained by \emph{surgery on $\ell$}.  A local picture of the process of obtaining $(\Sigma';\A',\n',\g')$ from $(\Sigma;\A,\n,\g)$ and $\ell$ is shown in Figure~\ref{fig:deccurve}.

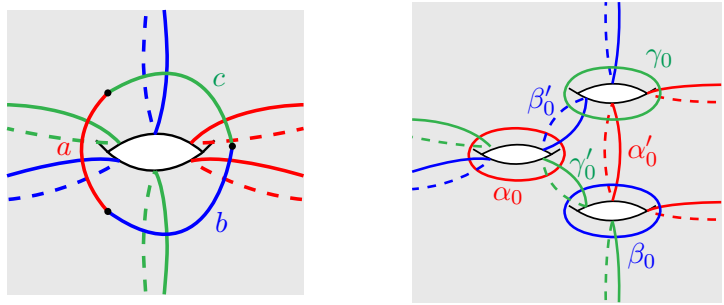
\begin{figure}[h!]
  \centering
  \includegraphics[width=.3\linewidth]{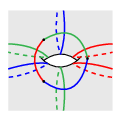} \qquad
    \includegraphics[width=.3\linewidth]{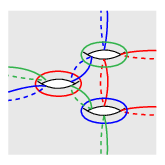}
\put (-198,35) {\textcolor{blue}{$b$}}
\put (-199,90) {\textcolor{ForestGreen}{$c$}}
\put (-258,64) {\textcolor{red}{$a$}}
\put (-42,23) {\textcolor{blue}{$\n_0$}}
\put (-42,63) {\textcolor{red}{$\A'_0$}}
\put (-93,47) {\textcolor{red}{$\A_0$}}
\put (-35,99) {\textcolor{ForestGreen}{$\g_0$}}
\put (-64,58) {\textcolor{ForestGreen}{$\g'_0$}}
\put (-80,82) {\textcolor{blue}{$\n'_0$}}
  \caption{At left, a decomposed curve $\ell$ in a trisection diagram $(\Sigma;\A,\n,\g)$.  At right, the diagram $(\Sigma';\A',\n',\g')$ obtained by surgery on $\ell$.}
	\label{fig:deccurve}
\end{figure}

A $\star$-trisection diagram is a tuple $(F;a,b,c)$ consisting of a compact surface $F$ and three sets of curves describing three compression bodies with positive boundary equal to $F$ such that, each pair of curves, is slide equivalent to a standard diagram; see Figure 10 of~\cite{AM}. A $\star$-trisection diagram describes a decomposition of compact 4-manifolds similar to a relative trisection. Please refer to~\cite{AM} for the precise definitions. The following %folllows from the work in~\cite{AM}:
is contained in the proof of Theorem 7.1 of~\cite{AM}:

\begin{theorem}\label{thm:roman}
Suppose $X$ admits a $(g;k_1,k_2,k_3)$-trisection $\T$ with diagram $(\Sigma;\A,\n,\g)$ and decomposed curve $\ell$, and let $(\Sigma';\A',\n',\g')$ be obtained by surgery on $\ell$.  Then $(\Sigma';\A',\n',\g')$ is a trisection diagram for a $(g+2; k_1,k_2,k_3)$-trisection $\T'$ of a 4-manifold $X'$, where $X'$ is obtained by surgery on $\ell$, viewed as a loop in $X$.
\end{theorem}

\begin{proof}
%\emph{Roman please fill in proof referring to~\cite{AM}.}
First, we note that $(\Sigma';\A',\n')$ admits two destabilizations, which can be realized by compressing along $\A'_0$ and $\n'_0$, the result of which yields $(\Sigma;\A,\n)$.  Similar statements hold for $(\Sigma';\n',\g')$ and $(\Sigma';\g',\A')$.  Thus, $\T'$ is a $(g+2;k_1,k_2,k_3)$-trisection diagram. 
%We now check that $X'$ is diffeomorphic to $X_\ell$. 

As shown in Figure 16 of~\cite{AM}, the tuple $(P; \alpha'_0, \beta'_0, \gamma'_0)$ is a $\star$-trisection diagram for $S^2\times D^2$. By construction, cutting $\Sigma'$ along $\alpha'_0\cup \beta'_0\cup \gamma'_0$ leaves us with a pair of pants $P$ and a surface $\widetilde{\Sigma}$ homeomorphic to a copy of $\Sigma$ with three small disks removed. Thus, the Pasting Lemma for $\star$-trisections~\cite[Rem 5.2]{AM} implies that $X'$ can be decomposed as the union $X=(S^2\times D^2)\cup Y$ for some 4-manifold $Y$. This new 4-manifold $Y$ is described by the $\star$-trisection diagram $(\widetilde{\Sigma}; \alpha\cup \{\alpha'_0\}, \beta\cup \{\beta'_0\}, \gamma\cup \{\gamma'_0\})$. In Section 6 of~\cite{AM} (specifically Remark 6.2), it was proven that the tuple $(\widetilde{\Sigma}; \alpha\cup \{\alpha'_0\}, \beta\cup \{\beta'_0\}, \gamma\cup \{\gamma'_0\})$ is a $\star$-trisection diagram for $X\setminus \ell$ where $\ell$ is the decomposed loop $a\cup b\cup c$. Hence, $X'$ can be obtained by removing a neighborhood of $\ell$ and gluing in a copy of $S^2 \X D^2$; that is, $X'$ is the result of loop surgery along $\ell$. 
%Let $\widetilde{\Sigma}$ be the result of removing small open disks around the endpoints of $a$, $b$, and $c$, and denote by $\alpha_0$, $\beta_0$, and $\gamma_0$ the boundaries of $\widetilde{\Sigma}$ so that $a\cap \widetilde{\alpha_0}=\emptyset$, $b\cap \widetilde{\beta_0}=\emptyset$, and $c\cap \widetilde{\gamma_0}=\emptyset$. In Section 6 of~\cite{AM}, especifically Remark 6.2, it was proven that the tuple $(\widetilde{\Sigma}; \alpha\cup \widetilde{\alpha_0}, \beta\cup \widetilde{\beta_0}, \gamma\cup \widetilde{\gamma_0})$ is a $\star$-trisection diagram for $X\setminus \ell$. %
%On the other hand, the tuple $(P; \alpha'_0, \beta'_0, \gamma'_0)$ is a $\star$-trisection diagram for $S^2\times D^2$; see Figure 16 of~\cite{AM}. 
%To end, using the Pasting Lemma for $\star$-trisections \cite[Rem 5.2]{AM}, we obtain a trisection diagram by gluing $\widetilde{\Sigma}$ with $P$  $\alpha'_0=$
\end{proof}

\begin{remark}
Meier's diagrams for trisections of spun lens spaces~\cite{meier} are examples of diagrams obtained from surgery on a loop in a genus-one trisection of $S^1\times S^3$.  An example is shown in Figure~\ref{fig:spun_star}.  At left, we see a decomposed curve in the genus-one trisection of $S^1 \X S^3$.  At right, the result of loop surgery on $\ell$ is handleslide equivalent to the trisection diagram of spun $\mathbb{RP}^3$ shown in Figure~\ref{fig:spunrp2}.
\end{remark}

\begin{figure}[h!]
  \centering
  \includegraphics[width=.4\linewidth]{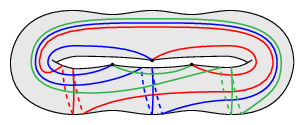} \qquad
    \includegraphics[width=.4\linewidth]{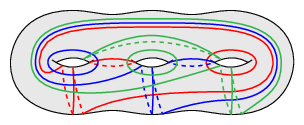}
  \caption{At left, a decomposed curve $\ell$ in a genus-one trisection diagram of $S^1\times S^3$.  At right, the diagram of spun $\mathbb{RP}^3$ obtained by surgery on $\ell$, which is equivalent to the diagram in Figure~\ref{fig:spunrp2}.}
	\label{fig:spun_star}
\end{figure}

Next, consider a trisection diagram $(\Sigma;\A,\n,\g)$ for a $(g;k_1,k_2,k_3)$-trisection $\T$.  We say that the trisection diagram contains a \emph{five-chain} $\{\g_2,\n_1,\A_1,\g_1,\n_2\}$ if
\[ \A_1 \cap \n_2 = \A_1 \cap \g_2 = \n_2 \cap \g_2 = \n_1 \cap \g_1 = \emp \quad \text{and} \quad |\g_2 \cap \n_1| = |\n_1 \cap \A_1| = |\A_1 \cap \g_1| = |\g_1 \cap \n_1| = 1,\]
and if one component $P$ of $\Sigma \setminus (\A_1 \cup \n_2 \cup \g_2)$ is a pair of pants (a thrice-punctured sphere).  See Figure~\ref{fig:5chain} for a general example.  Note that diagrams obtained by surgery on a loop (e.g. Figure~\ref{fig:deccurve} and~\ref{fig:spun_star}) always contain a five-chain, but the converse is not necessarily true.

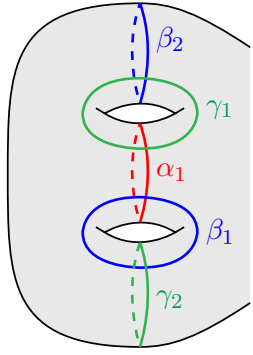
\begin{figure}[h!]
  \centering
  \includegraphics[width=.25\linewidth]{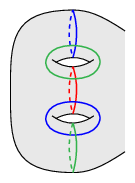}
  \put (-25,49) {\textcolor{blue}{$\n_1$}}
\put (-44,73) {\textcolor{red}{$\A_1$}}
\put (-25,97) {\textcolor{ForestGreen}{$\g_1$}}
\put (-44,122) {\textcolor{blue}{$\n_2$}}
\put (-44,25) {\textcolor{ForestGreen}{$\g_2$}}
  \caption{A five-chain $\{\g_2,\n_1,\A_1,\g_1,\n_2\}$ within a trisection diagram.}
	\label{fig:5chain}
\end{figure}

\begin{lemma}\label{lem:fiveslide}
Suppose $(\Sigma;\A,\n,\g)$ contains a \emph{five-chain} $\{\g_2,\n_1,\A_1,\g_1,\n_2\}$.  Possibly after handeslides, we may assume that $\A$ intersects $\text{int}(P)$ in some number of parallel seams, $\n$ intersects $\text{int}(P)$ a single seam coming from $\n_2$, and $\g$ intersects $\text{int}(P)$ in a single seam coming from $\g_2$.
\end{lemma}

\begin{proof}
In general, curves in $\A$, $\n$, and $\g$ (excluding curves in $\pd P$) meet $P$ in waves or seams.  Any wave in $\A \cap P$ can be slid over $\A_1$ and homotoped away from $P$, and similar statements hold for waves in $\n \cap P$ and $\g \cap P$.  Thus, after handleslides, we have assume that $\A$, $\n$, and $\g$ meet $\text{int}(P)$ in seams.  All $\n$-seams must be parallel to the one arc of $\n_2 \cap P$, and as such, each seam can be slid over $\n_2$ and then homotoped away from $P$.  Similarly, $\g$-seams are parallel to $\g_2 \cap P$ and can be slid and homotoped away from $P$, completing the proof.  See Figure~\ref{fig:fiveslide}.
\end{proof}

\begin{figure}[h!]
  \centering
  \includegraphics[width=.2\linewidth]{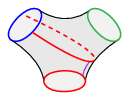} \raisebox{1cm}{$\longrightarrow$}
  \includegraphics[width=.2\linewidth]{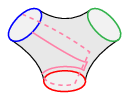} \qquad \quad 
  \includegraphics[width=.2\linewidth]{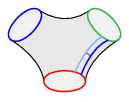}  \raisebox{1cm}{$\longrightarrow$}
  \includegraphics[width=.2\linewidth]{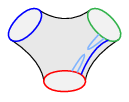}
  \caption{Examples of the handleslides described in Lemma~\ref{lem:fiveslide}}
	\label{fig:fiveslide}
\end{figure}

Now, suppose $(\Sigma;\A,\n,\g)$ contains a five-chain $\{\g_2,\n_1,\A_1,\g_1,\n_2\}$ such that $P$ meets $\A$, $\n$, and $\g$ as described in Lemma~\ref{lem:fiveslide}, and let $\A' = \A - \{\A_1\}$, $\n' = \n - \{\n_1\}$, and $\g' = \g - \{\g_1\}$, so that $\A'$, $\n'$, and $\g'$ are cut systems for the compressed surface $\Sigma_{\A_1}$.  We say that the Heegaard triple $(\Sigma_{\A_1};\A',\n',\g')$ is obtained from the original diagram by \emph{five-chain surgery}.  See Figure~\ref{fig:fivesurgery}.

\begin{lemma}\label{lem:fivechain}
If $(\Sigma;\A,\n,\g)$ is a $(g;k_1,k_2,k_3)$-trisection diagram containing a five-chain, and $(\Sigma_{\A_1};\A',\n',\g')$ is obtained by five-chain surgery, then $(\Sigma_{\A_1};\A',\n',\g')$ is a $(g-1;k_1,k_2,k_3+1)$-trisection diagram.
\end{lemma}

\begin{proof}
We need only check that cut systems pair to yield the required 3-manifolds.  Note that $(\Sigma_{\A_1};\A',\n')$ and $(\Sigma_{\A_1};\A',\g')$ are destabilizations of $(\Sigma;\A,\n)$ and $(\Sigma;\A,\g)$, and so the only pairing that needs to be checked is $(\Sigma_{\A_1};\n',\g')$.  By Lemma~\ref{lem:fiveslide}, we may assume that no curves in $\n$ meet $\g_1$ and no curves in $\g$ meet $\n_1$.  Similarly, any $\n$ curves (except $\n_1$) that intersect $\g_2$ can be slid over $\n_1$ to eliminate these intersection, and any $\g$ curves (except $\g_1)$ that intersect $\n_2$ can be slide over $\g_1$.  Thus, after handleslides, we may assume that $\n \cap \g_2 = \n_1 \cap \g_2$ and $\g \cap \n_2 = \g_1 \cap \n_2$.  The result of five-chain surgery then excises a genus-2 $S^3$ summand from $(\Sigma;\n,\g)$ and replaces it with a genus one summand corresponding to $S^1 \X S^2$, in which the resulting curves $\n_1$ and $\g_1$ become parallel.  In other words, $(\Sigma_{\A_1};\n',\g')$ determines a genus-$(g-1)$ splitting of $\#^{k_3+1} (S^1 \X S^2)$, as desired.
\end{proof}

In Figure~\ref{fig:fiveex}, we can verify that the result of five-chain surgery performed on the trisection diagram from Figure~\ref{fig:spunrp2} yields a $(2;1,1,2)$-trisection of $S^1 \X S^3$.

\begin{figure}[h!]
  \centering
    \includegraphics[width=.4\linewidth]{fig/spunrp2.eps} \qquad
  \includegraphics[width=.4\linewidth]{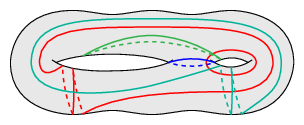}
  \caption{At left, a trisection diagram for $S_2$ containing a five-chain.  At right, the result of five-chain surgery, a $(2;1,1,2)$-trisection of $S^1 \X S^3$.}
	\label{fig:fiveex}
\end{figure}

By Lemma~\ref{lem:fivechain} above, we know that five-chain surgery on a trisection diagram for $X$ produces a new trisection diagram for some 4-manifold $X'$, and the following proposition uses decomposed curves to understand how $X$ and $X'$ are related.

\begin{proposition}\label{prop:surgery}
Suppose a trisection diagram for a 4-manifold $X$ contains a \emph{five-chain}.  Then the trisection diagram obtained by surgery on a five-chain corresponds to a 4-manifold $X'$ obtained by surgery on a 2-sphere in $X$.  Conversely, $X$ is obtained by surgery on a loop in $X'$.
\end{proposition}

\begin{proof}
We prove that $X$ is obtained by surgery on a loop in $X'$; the reverse follows from the discussion in Section~\ref{sec:prelim}.  Let $(\Sigma;\alpha,\beta,\gamma)$ be a trisection diagram for $X$ containing a five-chain $\{\gamma_2,\beta_1, \alpha_1, \gamma_1, \beta_2\}$ and let $(\Sigma_{\A_1};\alpha',\beta',\gamma')$ be the result of five-chain surgery, a trisection diagram for some 4-manifold $X'$ by Lemma~\ref{lem:fivechain}.  Let $D_1$ and $D_1'$ denote the scars of $\A_1$ in $\Sigma_{\A_1}$, where $D_1$ is contained in the annulus cobounded by $\beta_2$ and $\gamma_2$.  In $\Sigma_{\A_1}$, the curves $\beta_2$ and $\gamma_2$ induce arcs $b$ and $c$, respectively, from $D_1$ to $D_1'$.  Let $a$ be an arc in $D_1$ connecting $\pd b$ to $\pd c$.  Shrinking $D_1'$ to a point yields a decomposed curve $\ell$, as shown the top row of Figure~\ref{fig:fiveloop}.

By Theorem~\ref{thm:roman}, the result $(\Sigma'';\A'',\n'',\g'')$ of doing surgery on $\ell$ in $(\Sigma_{\A_1}; \A',\n',\g')$ is a trisection diagram for $X''$, where $X''$ is obtained by surgery on $\ell$ in $X'$, letting $\A_0$, $\n_0'$, $\g_0'$ be defined as in the definition of surgery on a decomposed curve.  Sliding $\n_0'$ over $\n_2$ (left over from the original five-chain) and sliding $\g_0'$ over $\g_2$ yields a curve $\n_* = \g_*$ that intersects $\A_0$ once.  Any arcs of $\A''$ that meet $\n_*$ can be slid over $\A_0$, showing that the trisection $\T''$ determined by $(\Sigma'';\A'',\n'',\g'')$ admits a destabilization.  These slides are shown in the bottom row of Figure~\ref{fig:fiveloop}.  Finally, destabilizing yields a diagram identical to the $(\Sigma;\A,\n,\g)$ in our hypotheses, implying that $X = X''$, and completing the proof.  
\end{proof}

\begin{figure}[h!]
  \centering
    \includegraphics[width=.25\linewidth]{fig/fivesurgery1.eps} \qquad
  \includegraphics[width=.25\linewidth]{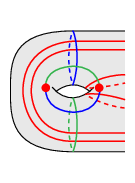} \qquad
   \includegraphics[width=.25\linewidth]{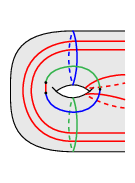}
     \put (-300,59) {\textcolor{blue}{$\n_1$}}
\put (-312,73) {\textcolor{red}{$\A_1$}}
\put (-300,105) {\textcolor{ForestGreen}{$\g_1$}}
\put (-312,128) {\textcolor{blue}{$\n_2$}}
\put (-312,32) {\textcolor{ForestGreen}{$\g_2$}}
\put (-202,80) {\textcolor{red}{$D_1$}}
\put (-159,75) {\textcolor{red}{$D'_1$}}
\put (-80,75) {\textcolor{red}{$a$}}
\put (-35,93) {\textcolor{ForestGreen}{$c$}}
\put (-35,51) {\textcolor{blue}{$b$}}\\
       \includegraphics[width=.25\linewidth]{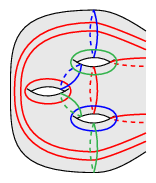} \qquad
    \includegraphics[width=.25\linewidth]{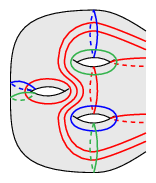}
    \put (-222,75) {\textcolor{red}{$\A_0$}}
\put (-205,43) {\textcolor{ForestGreen}{$\g_0'$}}
\put (-207,80) {\textcolor{blue}{$\n_0'$}}
\put (-115,62) {\textcolor{teal}{$\n_*$}}
  \caption{At top left, a diagram containing a five-chain.  At top center, the diagram induced by five-chain surgery, with scars from $\A_1$ and arcs from $\n_1$ and $\g_1$.  At top right, the decomposed curve $\ell$ resulting from five-chain surgery.  At bottom left, a diagram resulting from surgery from $\ell$.  At bottom right, the result of slides revealing a destabilization, which yields in the diagram at top left.}
	\label{fig:fiveloop}
\end{figure}

\begin{remark}
An alternative and more circuitous proof of Proposition~\ref{prop:surgery} can be obtained by converting the relevant trisection diagrams to Kirby diagrams as in~\cite{kepplinger} and replacing a 0-framed 2-handle with a dotted circle, corresponding to a 1-handle, the surgery operation in this setting.
\end{remark}

\section{Proof of the main theorem}\label{sec:proof}

In this section, we combine all of the ingredients we have assembled to give a proof of Theorem~\ref{thm:main}, the statement of which has two parts:  That every weakly reducible genus-three trisection is reducible or contains a five-chain, and that the corresponding 4-manifolds are $S_p$, $S_p'$, or a short list of standard simply-connected manifolds.

\begin{proof}[Proof of Theorem~\ref{thm:main}]
Suppose $X$ admits a weakly reducible genus-three trisection $\T$.  If $\T$ is reducible, then $X$ can be expressed as a connected sum $X' \# X''$, where $X'$ admits a genus-one trisection and $X''$ admits a genus-two trisection.  By the main theorem in~\cite{MZ2}, it follows that $X$ is $S^4$ or a connected sum of copies of $\pm \CP^2$, $S^1 \X S^3$, and $S^2 \X S^2$.

By Proposition~\ref{prop:dp}, either $\T$ is reducible, completing the proof, or without loss of generality, we have $\n_3 = \g_3$ and $|\A_1 \cap \n_1| = |\A_1 \cap \g_1| = 1$. By sliding $\beta_2$ and $\gamma_2$ over $\beta_1$ and $\gamma_1$, respectively, we can also suppose that $\n_2$ and $\g_2$ are disjoint from $\A_1$.  Let $Y$ be the 3-manifold determined by the genus-one Heegaard diagram $(\Sigma_{\A_1,\n_3};\n_2,\g_2)$, so that $Y$ is either $S^3$, $S^1 \X S^2$, or a lens space $L(p,q)$.  Additionally, suppose that $\T$ is a $(3;k_1,k_2,k_3)$-trisection, noting that the existence of $\n_3 = \g_3$ induces $k_3 = 1$.  Up to the symmetry of $\n$ and $\g$, this yields three cases to consider: $(k_1,k_2) = (1,1)$, $(0,1)$, or $(0,0)$.  Thus, considering both the possibilities for $Y$ and the possibilities for $(k_1,k_2)$, the proof involves examining nine cases.

Let $\A' = \{\A_2,\A_3\}$, $\n' = \{\n_2,\n_3\}$, and $\g' = \{\g_2,\g_3\}$, and consider the Heegaard triple $(\Sigma_{\A_1}; \A',\n',\g')$.  By construction, we have
\begin{enumerate}
\item $(\Sigma_{\A_1}; \A',\n')$ is a diagram for $\#^{k_2}(S^1 \X S^2)$,
\item $(\Sigma_{\A_1}; \A',\g')$ is a diagram for $\#^{k_1}(S^1 \X S^2)$, and
\item $(\Sigma_{\A_1}; \n',\g')$ is a diagram for $Y \# (S^1 \X S^2)$.
\end{enumerate}

First, suppose that $Y = S^3$.  Then $(\Sigma_{\A_1}; \A',\n',\g')$ is a $(2;k_1,k_2,1)$-trisection diagram for a genus-2 trisection $\T'$, which is standard by the main theorem in~\cite{MZ2}.  If $(k_1,k_2) = (1,1)$, then $\T'$ is a $(2;1)$-trisection, which can be expressed as the connected sum of a $(1,0)$-trisection of $\pm \CP^2$ and a $(1,1)$-trisection of $S^1 \X S^3$.  In particular, there is a non-separating curve $c$ in $\Sigma_{\A'}$ that bounds disks in each of $H_{\A'}$, $H_{\n'}$, and $H_{\g'}$.  Here we take care, because although we know immediately that $c$, considered as a curve in $\Sigma$, bounds a disk in $H_{\A}$  since $\Sigma_{\A_1}$ is obtained by compressing along a curve in $\A$, it does not immediately follow that $c$ also bounds disks in $H_{\n}$ or $H_{\g}$.

However, note that both $c$ and $\n_3 = \g_3$ are non-separating reducing curves for the Heegaard splitting of $S^1 \X S^2$ determined by $(\Sigma_{\A_1}; \n',\g')$.  By Lemma~\ref{lem:unique1}, it follows that $c$ is homotopic to $\n_3$ in $\Sigma_{\A_1}$, which in turn implies that $c$ becomes homotopic to $\n_3$ in $\Sigma$ after a sequence of handleslides of $c$ over $\A_1$.  Since both $c$ and $\A_1$ bound disks in $H_{\A}$, it follows that $\n_3$ also bounds a disk in $H_{\A}$, and thus $\n_3$ bounds a disk in all three handlebodies.  We conclude that $\T$ is reducible, as desired.

If $(k_1,k_2) = (0,1)$ or $(0,0)$, then the trisection $\T'$ is stabilized, and there are curves $c$ and $c'$ in $\Sigma_{\A_1}$ such that $c$ bounds a disk in $H_{\A'}$, $c'$ bounds disks in both of $H_{\n'}$ and $H_{\g'}$, and $|c \cap c'| = 1$.  By Lemma~\ref{lem:unique1}, we have that $c'$ and $\n_3$ are homotopic in $\Sigma_{\A'}$, and so $c$ is homotopic in $\Sigma_{\A'}$ to a curve $c^*$ that meets $\n_3$ in a single point.  Finally, $c$ becomes homotopic to $c^*$ in $\Sigma$ after some number of handleslides of $c$ over $\A_1$, and since both $c$ and $\A_1$ bound disks in $H_{\A}$, we have that $c^*$ bounds a disk in $H_{\A}$ as well.  Since $|c^* \cap \n_3| = 1$, we conclude that our original trisection $\T$ is stabilized and thus $\T$ is reducible.

Next, we move on to the case in which $Y = L(p,q)$.  Consider the Heegaard triple $(\Sigma_{\A_1}; \A',\n',\g')$.  In this case, Lemma~\ref{lem:tripexist} implies that it is not possible to have $(k_1,k_2) = (0,1)$.  If $(k_1,k_2) = (1,1)$, applying Proposition~\ref{prop:lens} to the Heegaard triple $(\Sigma_{\A_1}; \A',\n',\g')$ yields that there exists a non-separating curve $c \subset \Sigma_{\A_1}$ bounding disks in each of $H_{\A'}$, $H_{\n'}$, and $H_{\g'}$.  As above, $c$ also bounds a disk in $H_{\A}$ when we consider $c$ as a curve in $\Sigma$, but a priori we cannot assume $c$ bounds a disk in $H_{\n}$ or $H_{\g}$.  Nevertheless, Lemma~\ref{lem:unique2} asserts that the non-separating reducing curve for $(\Sigma_{\A_1}; \n',\g')$ is unique, and thus $c$ is homotopic to $\n_3$ in $\Sigma_{\A_1}$.  As above, this implies that $c$ becomes homotopic to $\n_3$ in $\Sigma$ after a sequence of handleslides of $c$ over $\A_1$, and thus $\n_3$ bounds a disk in all three handlebodies, and $\T$ is reducible.

Similarly, if $(k_1,k_2) = (0,0)$, our Heegaard triple $(\Sigma_{\A_1}; \A',\n',\g')$ satisfies the hypotheses of Theorem~\ref{thm:triple}, which implies that there are curves $c_1$ bounding a disk in $H_{\A'}$ and $c_2$ bounding disks in both $H_{\n'}$ and $H_{\g'}$ such that $|c_1 \cap c_2| = 1$.  By Lemma~\ref{lem:unique2}, we have that $c_2$ and $\n_3$ are homotopic in $\Sigma_{\A'}$, and the proof proceeds as above to conclude that $\T$ is stabilized.

Finally, we consider the case in which $Y = S^1 \X S^2$.  Similar the the previous case,  Lemma~\ref{lem:tripexist} applied to $(\Sigma_{\A_1}; \A',\n',\g')$ yields that $(k_1,k_2)$ cannot equal $(0,1)$.  The only remaining sub-cases are $(k_1,k_2) = (1,1)$ and $(k_1,k_2) = (0,0)$.  We will show that in these two cases, we can find a five-chain $\{\g_2^*,\n_1^*, \A_1,\g_1^*, \n_2\}$ using the data from the original trisection.  Let $\n^* = \{\n_1,\n_2\}$ and $\g^* = \{\g_1,\g_2\}$, and consider the genus-two Heegaard splitting for $S^3$ given by $(\Sigma_{\n_3};\n^*,\g^*)$.  Since $\A_1 \cap \n_3 = \emp$, we can view $\A_1$ as a framed knot in this copy of $S^3$, with framing determined by $\Sigma_{\n_3}$.  Thus, we can perform surface-framed Dehn surgery along $\A_1$, which can be accomplished by attaching a 3-dimensional 2-handle to $H_{\n^*}$ along $\A_1$, attaching a 3-dimensional 2-handle to $H_{\g^*}$ along $\A_1$, and gluing the results together.  Since $|\A_1 \cap \n_1| = |\A_1 \cap \g_1| = 1$, the resulting pieces are two solid tori determined by $\n_2$ and $\g_2$ with boundary $\Sigma_{\{\A_1,\n_3\}}$.  In other words, the 3-manifold $Y = S^1 \X S^2$ is obtained from $S^3$ by surgery on $\A_1$.  By Gabai's Property R~\cite{Gabai2}, we can conclude that $\A_1$ is zero-framed unknot in this copy of $S^3$.

Now, push $\A_1$ into the interior of $H_{\n^*}$ and let $C = H_{\n^*} \setminus \A_1$, so that $C$ is a compression-body and $C \cup_{\Sigma_{\n_3}} H_{\g^*}$ is a genus-two Heegaard splitting of the exterior of $\A_1$ in $S^3$, a solid torus, which must be stabilized (see, for instance,~\cite{lei}).  Since $\n_2$ was chosen to be disjoint from $\A_1$, and $C$ has a unique non-separating compressing disk by Lemma~\ref{lem:comp1} (namely, the one bounded by $\n_2$), it follows that $\n_2$ is primitive in $H_{\g^*}$.

Next, define $H_{\A_1,\n_2}$ to be the handlebody determined by $\{\A_1,\n_2\}$, noting that $H_{\A_1,\n_2} \cup_{\Sigma_{\n_3}} H_{\g^*}$ is a genus-two Heegaard splitting for the manifold obtained by surgery on $\A_1$; that is, $S^1 \X S^2$.  As established above, both $\A_1$ and $\n_2$ are primitive in $H_{\g^*}$, and thus by Lemma~\ref{lem:primitive}, there is a curve $\gamma_1^*$ bounding a disk in $H_{\g^*}$ that intersects each of $\A_1$ and $\n_2$ in a single point.  We have now found three out of the five curves that will form our five-chain.

A priori, it may be possible that $\n_1 \cap \gamma_1^* \neq \emp$.  However, $|\gamma_1^* \cap \n_2| = 1$, and so there is a series of slides of $\n_1$ over $\n_2$ along sub-arcs of $\gamma_1^*$ that produce a new curve $\n_1^*$ such that $\gamma_1^* \cap \n_1^* = \emp$.  In addition, $|\gamma_1^* \cap \A_1| = 1$, and so these slides can be chosen to take place along sub-arcs of $\gamma_1$ that do not intersect $\A_1$; hence, the result $\n_1^*$ of all of these handleslides satisfies $|\n_1^* \cap \A_1| = |\n_1 \cap \A_1| = 1$.  See Figure~\ref{fig:slides}. The curves $\gamma^*_1$ and $\beta^*_1$ lie in $\Sigma_{\beta_3}$ so they can be though of as curves in $\Sigma$ disjoint from $\beta_3$. Hence $\{\beta^*_1, \alpha_1, \gamma^*_1, \beta_2\}$ are four-fifths of our desired five-chain.

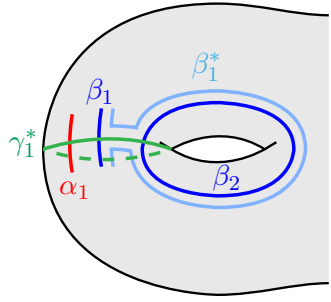
\begin{figure}[h!]
  \centering
  \includegraphics[width=.3\linewidth]{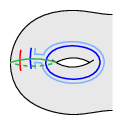}
    \put (-55,52.5) {\textcolor{blue}{$\n_2$}}
\put (-132,66) {\textcolor{ForestGreen}{$\g^*_1$}}
\put (-113,49) {\textcolor{red}{$\A_1$}}
\put (-103,85) {\textcolor{blue}{$\n_1$}}
\put (-63,95) {\textcolor{CornflowerBlue}{$\n^*_1$}}
  \caption{The curve $\n_1$ can be converted to a light blue curve $\n_1^*$ disjoint from $\g_1^*$ by sliding over $\n_2$.}
	\label{fig:slides}
\end{figure}

Having established four of five curves in our five-chain, let $\g_2^*$ be a curve in $\Sigma$ such that $\{\g_1^*,\g_2^*,\g_3\}$ is a cut system for $H_{\g}$, $\g_2^* \cap \A_1 = \emp$, and $\g_2^*$ intersects $\n_1^* \cup \n_2$ minimally among all such choices.  Let $\Sigma_* = \Sigma_{\n_3} \setminus (\A_1 \cup \g^*_1)$, so that $\Sigma_*$ is a torus with one boundary component, with $H_1(\Sigma_*) = \Z \oplus \Z$.  Note that $\Sigma_*$ contains two scars from compressing along $\n_3$, and $\pd \Sigma_*$ is a curve made up of two arcs coming from $\A_1$ and two arcs coming from $\g^*_1$.  Since $|\n^*_1 \cap \A_1| = 1$ and $\n_1^* \cap \g_1^* = \emp$, we have that $\n_1^* \cap \Sigma_*$ is a single essential arc.  Similarly, $|\n_2 \cap \g^*_1| = 1$ and $\n_2 \cap \A_1 = \emp$, so $\n_2 \cap \Sigma_*$ is another essential arc.  We parameterize $H_1(\Sigma_*)$ so that a $(0,1)$-curve is disjoint from $\n^*_1 \cap \Sigma_*$ and a $(1,0)$-curve is disjoint from $\n_2 \cap \Sigma_*$.  See Figure~\ref{fig:sigmastar}.

\begin{figure}[h!]
  \centering
  \includegraphics[width=.3\linewidth]{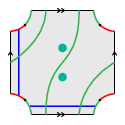} \qquad
    \includegraphics[width=.3\linewidth]{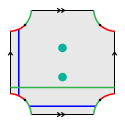} \qquad
            \put (-225,24) {\textcolor{blue}{$\n_2$}}
\put (-257,125) {\textcolor{ForestGreen}{$\g^*_1$}}
\put (-203,46) {\textcolor{ForestGreen}{$\g^*_2$}}
\put (-290,95) {\textcolor{red}{$\A_1$}}
\put (-262,58) {\textcolor{blue}{$\n^*_1$}}
        \put (-70,24) {\textcolor{blue}{$\n_2$}}
\put (-102,125) {\textcolor{ForestGreen}{$\g^*_1$}}
\put (-50,46) {\textcolor{ForestGreen}{$\g^*_2$}}
\put (-135,95) {\textcolor{red}{$\A_1$}}
\put (-107,68) {\textcolor{blue}{$\n^*_1$}}
  \caption{At left, a priori, $\g_2^*$ is a $(p,q)$-curve in $\Sigma_*$.  At right, we conclude that $\g_2^*$ must be a $(1,0)$-curve.}
	\label{fig:sigmastar}
\end{figure}
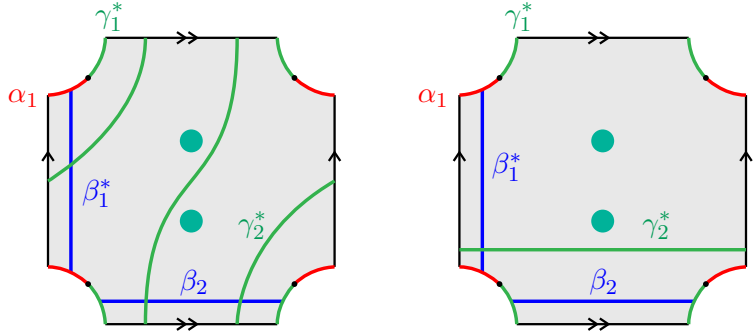

By construction, $\g^*_2$ is a curve in $\Sigma_*$, and so it is a $(p,q)$-curve for some integers $p$ and $q$.  We may assume that $|\g^*_2 \cap \n^*_1| = |p|$ and all intersection points have the same sign; if not, $\g^*_2$ and $\n^*_1$ bound a bigon in $\Sigma_*$, and after sliding $\g^*_2$ over $\g_3$ if necessary, we can decrease $|\g^*_2 \cap \n^*_1|$, contradicting our assumption of minimality.  By a similar argument, $|\g^*_1 \cap \n_2| = |q|$ and all intersection points have the same sign. Thus, the intersection matrix $M = M(\{\n^*_1,\n_2\},\{\g^*_1,\g^*_2\})$ is equal to 
\[ M= \begin{bmatrix}

0 & \pm p \\
1 & \pm q
\end{bmatrix}.\]
Observe that, $\{\n_1^*,\n_2\}$ is a cut system for $H_{\n^*}$ and $\{\g^*_1,\g^*_2\}$ is a cut system for $H_{\g^*}$, where $H_{\n^*} \cup_{\Sigma_{\n_3}} H_{\g^*} = S^3$, so that $|\det(M)| = |p| = 1$.  Finally, by the same reasoning as above, $(\Sigma_{\{\A_1,\n_3\}};\n_2,\g_2^*)$ is a Heegaard diagram for the 3-manifold obtained by surgery on $\A_1$ in $H_{\n^*} \cup_{\Sigma_{\n_3}} H_{\g^*}$, which we have shown is $S^1 \X S^2$.  It follows that $\n_2$ and $\g^*_2$ must intersect algebraically zero times, implying that $q = 0$, and so $\g^*_2$ is a $(\pm 1,0)$-curve in $\Sigma_*$.  We conclude that $|\g^*_2 \cap \n_1^*| = 1$ and $\g^*_2 \cap \n_2 = \emp$, as shown at right in Figure~\ref{fig:sigmastar}.  If necessary, slide $\g_2^*$ over $\n_3$ so that $\Sigma \setminus (\A_1 \cup \n_2 \cup \g_2^*)$ has a pair of pants component.  We conclude that $(\g^*_2,\n^*_1,\A_1,\g^*_1,\n_2)$ is a five-chain.

To complete the proof, let $\T'$ be the trisection obtained by five-chain surgery on $\T$.  By Lemma~\ref{lem:fivechain}, the trisection $\T'$ is a $(2;k_1,k_2,2)$-trisection of a 4-manifold $X'$, where the main theorem from~\cite{MZ2} implies that $\T'$ is standard.  If $(k_1,k_2) = (1,1)$, then $X'$ is diffeomorphic to $S^1 \X S^3$, and by Proposition~\ref{prop:surgery}, $X$ is obtained by surgery on a loop in $S^1 \X S^3$, implying that $X$ is diffeomorphic to $S_p$ or $S_p'$.  On the other hand, if $(k_1,k_2) = (0,0)$, then $X'$ is diffeomorphic to $S^4$, and again by Proposition~\ref{prop:surgery}, we have that $X$ is obtained by surgery on a loop in $S^4$, so that $X$ is either $S^2 \X S^2$ or $\CP^2 \# \overline{\CP}^2$, completing the proof.
 \end{proof}
 
 \section{Dependent triples}\label{sec:depend}
 
We can extend the proof in the previous section to trisections containing dependent triples.  Recall that a trisection $\T$ admits a \emph{dependent triple} if there are pairwise disjoint non-separating curves $\A_1$, $\n_1$, and $\g_1$ bounding disks in $H_{\A}$, $H_{\n}$, and $H_{\g}$, respectively, such that $[\A_1]$, $[\n_1]$, and $[\g_1]$ are linearly dependent in $H_1(\Sigma)$.  When $g(\Sigma) = 3$, there are three different ways up to homeomorphism that a dependent triple can occur:
\begin{enumerate}
\item Two of the curves are homotopic.
\item Two of the curves are homologous but are not homotopic.  In this case, the two curves cut $\Sigma$ into two twice-punctured tori.
\item No two curves are homologous.  In this case, $\Sigma \setminus (\A_1 \cup \n_1 \cup \g_1)$ has two components, where one component is a thrice-punctured torus and the other component is a pair of pants.
\end{enumerate}
These possibilities are shown in Figure~\ref{fig:triple}.

\begin{figure}[h!]
  \centering
  \includegraphics[width=.3\linewidth]{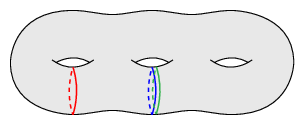} \quad
    \includegraphics[width=.3\linewidth]{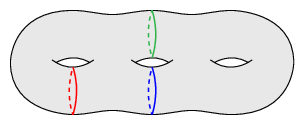} \quad
      \includegraphics[width=.3\linewidth]{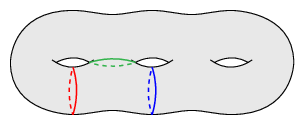}
  \caption{Up to symmetry and homeomorphism, the three ways a dependent triple can occur in $\Sigma$}
	\label{fig:triple}
\end{figure}

We need a lemma about trisection genus before we prove Theorem~\ref{thm:main2}.  Recall that $g(X) \geq \chi(X) - 2 + 3 \text{rk}(\pi_1(X))$\cite{CT}.

\begin{lemma}\label{lem:sum}
Suppose $X_1$ and $X_2$ are two closed 4-manifolds such that $g(X_i) = \chi(X_i) - 2 + 3 \text{rk}(\pi_1(X_i))$.  Then $g(X_1 \# X_2) = g(X_1) + g(X_2)$.
\end{lemma}

\begin{proof}
Certainly $g(X_1 \# X_2) \leq g(X_1) + g(X_1)$.  Noting that $\chi(X_1 \# X_2) = \chi(X_1) + \chi(X_2) - 2$, we have
\begin{eqnarray*}
g(X_1 \# X_2) &\geq& \chi(X_1 \# X_2) - 2 + 3 \text{rk}(\pi_1(X_1 \# X_2)) \\
&=& (\chi(X_1) + \chi(X_2) - 2) - 2 + 3(\text{rk}(\pi_1(X_1) + \pi_1(X_2))) \\
&=& \chi(X_1) - 2 + 3 \text{rk}(\pi_1(X_1)) + \chi(X_2) - 2 + 3 \text{rk}(\pi_1(X_2)) \\
&=& g(X_1) + g(X_2).
\end{eqnarray*}
\end{proof}

Notably, the manifolds $\pm \CP^2$, $S^2 \X S^2$, $S^1 \X S^3$, $S_p$, and $S_p'$ satisfy the hypotheses of Lemma~\ref{lem:sum}.

\begin{proof}[Proof of Theorem~\ref{thm:main2}]
Suppose that $X$ admits a genus-three trisection $\T$ containing a dependent triple $(\A_1,\n_1,\g_1)$, which falls into one of the three cases noted above.  In case (1), $\T$ is weakly reducible, and the conclusion follows from Theorem~\ref{thm:main}.  In case (2), we suppose without loss of generality that $\n_1$ and $\g_1$ cut $\Sigma$ into two twice-punctured tori.  In this case, the curves $\n_1$ and $\g_1$ are mutually separating, and so by applying Lemma~\ref{lem:weak2} to the Heegaard splitting $H_{\n} \cup H_{\g}$, we can conclude that either $\n_1$ bounds a disk in $H_{\g}$ or $\g_1$ bounds a disk in $H_{\n}$.  Thus, the trisection $\T$ is again weakly reducible, and we apply Theorem~\ref{thm:main}.

The third case requires more work.  Suppose that $\A_1$, $\n_1$, and $\g_1$ cobound a pair of pants $P$ in $\Sigma$.  If any of the curves bounds a disk in another handlebody, then the splitting is weakly reducible and we can apply Theorem~\ref{thm:main}.  Thus, suppose this does not occur.  Applying Lemma~\ref{lem:weak1} to the Heegaard splitting $H_{\A} \cup H_{\n}$, we can suppose without loss of generality that there is some $\n_2$ bounding a disk in $H_{\n}$ that meets $\A_1$ once and is disjoint from $\n_1$.  In this case, $\n_2 \cap P$ must be a single seam and some number of waves based in $\g_1$, which we can remove by handleslides of $\n_2$ over $\n_1$.  Thus, after slides we suppose that $\n_2 \cap P$ is a single seam and $|\n_2 \cap \g_1| = 1$.

Next, we apply Lemma~\ref{lem:weak1} to the Heegaard splitting $H_{\A} \cup H_{\g}$.  Either there is some $\g_2$ bounding a disk in $H_{\g}$ that meets $\A_1$ once and is disjoint from $\g_1$, or there is some $\A_2$ bounding a disk in $H_{\A}$ that meets $\g_1$ once and is disjoint from $\A_1$.  In the former case, we may suppose after slides that $\g_2 \cap P$ is a single seam and $|\g_2 \cap \n_1| = 1$.  In the latter case, we may suppose after slides that $\A_2 \cap P$ is a single seam and $|\A_2 \cap \n_1| = 1$.  These two possibilities are symmetric up to permutation of $\A$, $\n$, and $\g$, and so we suppose that there is some $\g_2$ bounding a disk in $H_{\g}$ that meets $\A_1$ once and is disjoint from $\g_1$, with $\g_2 \cap P$ a single seam and $|\g_2 \cap \n_1| = 1$.

We almost have a five-chain, but a priori it may be the case that $\n_2 \cap \g_2 \neq \emp$.  
However, since $|\beta_2\cap \gamma_1|=1$, we can find slides of $\gamma_2$ over $\gamma_1$ along arcs of $\beta_2$ that do not pass through $P$ to replace $\gamma_2$ with a curve $\gamma_2^*$ such that $\{\g_1,\n_2, \A_1,\g_2^*, \n_1\}$ is a five-chain.
%However, since $|\n_1 \cap \g_2| = 1$, we can find slides of $\g_1$ over $\g_2$ along arcs of $\n_1$ that do not pass through $P$ to replace $\g_2$ with a curve $\g_2^*$ such that $\{\g_1,\n_2, \A_1,\g_2^*, \n_1\}$ is a five-chain.  
See Figure~\ref{fig:moreslide}.

\begin{figure}[h!]
  \centering
  \includegraphics[width=.4\linewidth]{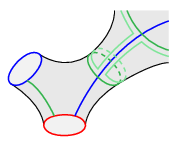}
        \put (-99,59) {\textcolor{blue}{$\n_2$}}
\put (-97,92) {\textcolor{ForestGreen}{$\g_1$}}
\put (-45,123) {\textcolor{ForestGreen}{$\g_2$}}
\put (-67,104) {\textcolor{green}{$\g^*_2$}}
\put (-115,21) {\textcolor{red}{$\A_1$}}
\put (-151,75) {\textcolor{blue}{$\n_1$}}
  \caption{Slides of $\g_2$ over $\g_1$ along arcs of $\n_2$ can be used to create a five-chain $\{\g_1,\n_2, \A_1,\g_2^*, \n_1\}$.}
	\label{fig:moreslide}
\end{figure}
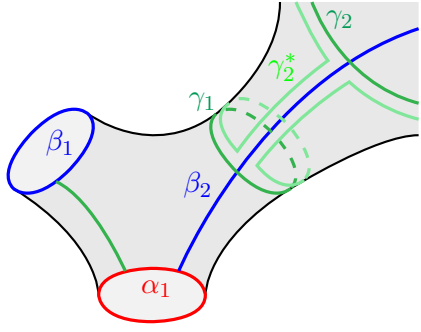

It follows from Lemma~\ref{lem:fivechain} and Proposition~\ref{prop:surgery} that $X$ is obtained by surgery on a loop $\ell$ in a manifold $X'$ admitting a $(2;k_1,k_2,k_3)$ trisection, where $k_1,k_2 \leq 1$.  By the main theorem from~\cite{MZ2}, we have that $X'$ is $S^4$, $S^2 \X S^2$, or a connected sum of at most two copies of $\pm \CP^2$ and/or $S^1 \X S^3$, with at most one $S^1 \X S^3$ summand.  If $\ell$ is null-homotopic in $X'$, then $\ell$ is isotopic into a 4-ball, in which case $X$ is a connected sum of copies of $S^2 \X S^2$, $\pm \CP^2$, and/or $S^1 \X S^3$.  Otherwise, $[\ell]$ is non-trivial in $\pi_1(X')$, from which it follows that $X'$ is either $S^1 \X S^3$ or $S^1 \X S^3 \# \pm \CP^2$, and thus $X$ can be $S_p$, $S_p'$, $S_p \# \pm \CP^2$, or $S'_p \# \pm \CP^2$, where $p \geq 1$.  If $p>1$, then by Lemma~\ref{lem:sum}, we have that $g(S_p \# \pm \CP^2) = (S_p' \# \pm \CP^2) = 4$, contradicting our assumption that $X$ admits a genus-three trisection.  We conclude $X$ is diffeomorphic to $S_p$, $S_p'$, or a connected sum of copies of $\pm \CP^2$ and $S^2 \X S^2$, completing the proof.
\end{proof}

 \section{Five-chain creation, speculations, and further questions}\label{sec:question}
 
 In the final section, we discuss an operation that reverses five-chain surgery, which we call \emph{five-chain creation}, and we include some problems for further investigation.  Consider a trisection diagram $(\Sigma;\A,\n,\g)$ for a $(g;k_1,k_2,k_3)$-trisection $\T$ of $X$, in which $k_3 \geq 1$ and $\n_1 = \g_1$, and choose a decomposed curve $\ell \subset \Sigma$ such that $|\ell \cap \n_1| = 1$ and $\ell \cap \n =\ell \cap \g = \ell \cap \n_1$.  Pushing $\g_1$ off $\n_1$, let $A \subset \Sigma$ denote the annulus cobounded by $\n_1$ and $\g_1$, and break $\ell$ into three arcs $a \cup b \cup c$ so that $a \subset A$, $b$ meets $\g_1$ in a single point, and $c$ meets $\n_1$ in a single point, as in the top right of Figure~\ref{fig:fiveloop}.
 
Let $(\Sigma'';\A'',\n'',\g'')$ be the trisection diagram obtained by surgery on $\ell$.  As in the proof of Proposition~\ref{prop:surgery}, the trisection $\T''$ correspond to $(\Sigma'';\A'',\n'',\g'')$ is stabilized and can be destabilized to a $(g+1;k_1,k_2,k_3-1)$-trisection diagram $(\Sigma;\A',\n',\g')$, which contains a five-chain, as shown in Figure~\ref{fig:fiveloop}.  We say that $(\Sigma;\A',\n',\g')$ is related to $(\Sigma;\A,\n,\g)$ by \emph{five-chain creation}, so that five-chain creation is the inverse operation to five-chain surgery.

The proof of Theorem~\ref{thm:main} reveals that weakly reducible genus-three trisections can be obtained by five-chain creation on a $(2;0,0,2)$-trisection of $S^4$ or a $(2;1,1,2)$-trisection of $S^1 \X S^3$.  Although these trisections are standard by Theorem~\ref{thm:MZ}, the starting point for five-chain creation is a trisection \emph{diagram}, not the trisection itself.  Both of the trisections above have the properties that $k_3 = g$, and so $H_{\n} = H_{\g}$.  Thus, interesting diagrams arise by starting with a genus-two Heegaard diagram $(\Sigma;\A,\n)$ for $S^3$ or $S^1 \X S^2$ to induce a trisection diagram $(\Sigma;\A,\n,\n)$ for $S^4$ or $S^1 \X S^3$, respectively.  In this case, the Heegaard diagram $(\Sigma;\A,\n)$ and choice of $\n_1$ induces a loop $\ell$ by picking a curve in $\Sigma$ dual to $\n_1$ and disjoint from $\n_2$.  Pushing $\ell$ slightly into $H_{\n}$ and then drilling it out yields a compression-body $C$, and thus $H_{\A} \cup_{\Sigma} C$ is a genus-two Heegaard splitting of $\ell$, considered as a knot in $S^3$ or $S^1 \X S^2$.  In other words, $\ell$ is a \emph{tunnel number one} knot in $S^3$ or $S^1 \X S^2$, and a core of $C$ is an \emph{unknotting tunnel} for $\ell$.  For more details about tunnel number and unknotting tunnels, see~\cite{tunnel}, for instance.

Reversing the process, given a genus-two Heegaard surface $\Sigma$ for a tunnel number one knot $K$ in $S^3$ or $S^1 \X S^2$, the infinitely many different ways to isotope $K$ into $\Sigma$ are determined by a choice of integral framing of $K$.  Thus, we have proved

\begin{proposition}\label{prop:tunnel}
Suppose that $\tau$ is an unknotting tunnel for a knot $K$ in $S^3$ or $S^1 \X S^2$, and $\lambda$ is an integral framing of $K$.  Then $(K,\tau,\lambda)$ induces a weakly reducible $(3;0,0,1)$- or $(3;1)$-trisection, respectively.  Moreover, every irreducible, weakly reducible genus-3 trisection is obtained in this way.
\end{proposition}

This leads naturally to a generalization of Question 4.3 from~\cite{meier}. Compare with Question 7.1 of~\cite{AM}.

\begin{question}
Are the weakly reducible trisections induced by five-chain creation standard?  In other words, if $[K] = [K']$ in $\pi_1(S^4)$ or $\pi_1(S^1 \X S^3)$ and $\lambda \equiv \lambda' \mod 2$, do $(K,\tau,\lambda)$ and $(K',\tau',\lambda')$ induce diffeomorphic trisections?
\end{question}

A major open problem in trisection theory is whether there exists a non-standard trisection of $S^4$; that is, whether every trisection of $S^4$ is a(n unbalanced) stabilization of the genus-zero trisection.  The construction of $\T$ from $(K,\tau,\lambda)$ appears to be sensitive to the isotopy class of $K$ in $S^3$ or $S^1 \X S^2$, despite the fact that two choices $K$ and $K'$ become isotopic in $S^4$ or $S^1 \X S^3$, provided they determine the same element of $\pi_1(S^1 \X S^3)$.  By choosing a complicated tunnel number one knot $K$ in $S^1 \X S^3$ such that $[K]$ generates $\pi_1(S^1 \X S^3)$, we can induce a $(3;1)$-trisection of $S^4$ that is a strong candidate for being non-standard.  Unfortunately, current techniques for obstructing reducibility of non-minimal trisections are limited; see Section 5 of \cite{MZDehn} for further discussion.

\begin{question}
Do there exist genus-three trisections that are not weakly irreducible?
\end{question}

Following the language from 3-manifold theory, we call such trisections \emph{strongly irreducible}.  Some potentially strongly irreducible $(3;1)$-trisections can be obtained by lifting the 4-bridge trisection of the $m$-twist spun knot $\mathcal{S}_m(K)$, where $K$ is a 2-bridge knot in $S^3$, to the double cover $X$ of $S^4$ branched over $\mathcal{S}_m(K)$.  This bridge trisection is shown in Figure 22 of~\cite{MZB}.  When $k$ is odd, $X \cong S^4$, while if $k$ is even, $X \equiv S_p$ for some $p$~\cite{twist}, and thus these genus-three trisections still satisfy Conjecture~\ref{conj:meier}, despite the possibility that they may not be weakly reducible.  As above, this question begs for new techniques with which to obstruct trisections from having certain geometric properties.

\bibliographystyle{amsalpha}
\bibliography{genusthreebib}

\end{document}